\documentclass[a4paper,12pt]{amsart}
\usepackage{amsfonts,amssymb,amscd}

\usepackage{epsf}
\usepackage{epsfig}
\usepackage{graphicx}

\newtheorem{teo}{Theorem}[section]
\newtheorem{prop}[teo]{Proposition}
\newtheorem{lema}[teo]{Lemma}
\newtheorem{obs}[teo]{Remark}
\newtheorem{defnc}[teo]{Definition}
\newtheorem{coro}[teo]{Corollary}

\newcommand{\C}{{\mathbb C}}
\newcommand{\R}{{\mathbb R}}

\newcommand{\Z}{{\mathbb Z}}
\newcommand{\N}{{\mathbb N}}

\newcommand{\fol}{{\mathcal F}}

\newcommand{\diff}{{{\rm Diff}\, ({\mathbb C}, 0)}}
\newcommand{\Hol}{{{\rm Hol}\, ({\mathbb C}, 0)}}

\newcommand{\diffalpha}{{{\rm Diff}_{\alpha} ({\mathbb C}, 0)}}

\begin{document}

\title[Stabilizers of pseudogroups on $(\C ,0)$]{Cyclic stabilizers and infinitely many hyperbolic orbits for pseudogroups on $(\C ,0)$}

\author{Julio C. Rebelo \hspace{0.2cm} \& \hspace{0.2cm} Helena Reis}
\address{}
\thanks{}

\begin{abstract}
Consider a pseudogroup on $(\C,0)$ generated by two local diffeomorphisms having analytic conjugacy classes {\it a priori}\,
fixed in $\diff$. We show that a generic pseudogroup as above is such that every point has (possibly trivial) cyclic stabilizer.
It also follows that these generic groups possess infinitely many hyperbolic
orbits. This result possesses several applications to the topology of leaves of foliations and we shall explicitly
describe the case of nilpotent foliations associated to Arnold's singularities of type $A^{2n+1}$.
\end{abstract}

\maketitle

\noindent \hspace{0.9cm} {\small Key-words: pseudogroups on $(\C, 0)$  -  fixed points  -  cyclic stabilizers}

\bigskip

\noindent \hspace{0.9cm} {\small AMS-Classification:  37C85  -  37F75 -  34M35}

\section{Introduction}

This paper provides answers to two types of well-known problems about pseudogroups on $(\C,0)$. For the convenience of the reader,
this Introduction will begin with a detailed presentation of them. The first problem has to do with the existence
of hyperbolic fixed points for {\it certain elements}\, of pseudogroups generated by local diffeomorphisms fixing
$0 \in \C$. The reader is reminded that a fixed point
$p$ for a local diffeomorphism $f$ between open sets of $\C$ is said to be {\it hyperbolic}\, if and only if
$\Vert f' (p) \Vert \neq 1$. In particular these fixed points are necessarily isolated, i.e. $f$ cannot coincide
with the identity on a neighborhood of a hyperbolic fixed point $p$.
The existence of ``generic'' pseudogroups for which there are infinitely many points $p_i$,
$p_i \neq 0$ for every $i$, such that each $p_i$ is an {\it hyperbolic fixed point} for a certain element $h_i$ belonging to
the pseudogroup in question was first singled out by Y. Il'yashenko in his seminal paper \cite{ilyaso}. The meaning
of these fixed points for the topology of the leaves of the corresponding foliations is explained in
\cite{ilyaso}, \cite{ilyaso2}. The more recent paper \cite{scherbaRO} improves Il'yashenko's
result on the existence of the fixed points in question.
Another result of similar nature was
established in \cite{gomez-mont} and applies to certain pseudogroups obtained
through ``perturbations'' of pseudogroups generated by local diffeomorphisms fixing $0 \in \C$ (note that these ``perturbed''
pseudogroups need no longer have the origin as a common fixed point).

Among the most interesting cases where the existence of hyperbolic fixed points have been considered, there
is the case of pseudogroups generated by local diffeomorphisms fixing $0 \in \C$ and such that
the associated {\it group of germs at}\, $0 \in \C$ is not solvable. Here the
existence of elements in the corresponding pseudogroup exhibiting {\it hyperbolic fixed points}\, has been
known since \cite{scherba}, \cite{nakai} and \cite{Frank1}, see also \cite{frank} for a comprehensive discussion
of these non-solvable pseudogroups. However the question on
whether or not these pseudogroups exhibit more than {\it one single orbit}\, of hyperbolic
fixed points, at least in the case of ``typical'' pseudogroups, has remained open. This question deserves to be elaborated further.

Suppose $p \neq 0$ happens to be a hyperbolic fixed point for some element $h$ of a given pseudogroup. Then the orbit
of $p$ is fully constituted by points that are hyperbolic fixed points for suitable conjugates of $h$ in the pseudogroup
in question. The problem that is naturally raised in this context consists of deciding whether the pseudogroup
contains two or more hyperbolic fixed points whose orbits are {\it mutually disjoint}. This question has a natural motivation
going back to foliations: in fact, whereas the
existence of an ``orbit of hyperbolic fixed points'' implies the existence
of a leaf carrying hyperbolic holonomy, it falls short from ensuring the existence of ``many leaves'' with hyperbolic holonomy.
The existence of these leaves can only be asserted if hyperbolic fixed points with mutually disjoint orbits is guaranteed.

The other problem concerning $(\C,0)$ that will be tackled in this paper has a more typical ``generic nature''.
It concerns the existence of points that are fixed by ``more than one element in the pseudogroup''. More precisely,
if $p \in \C$ is fixed by some element $h$, then $p$ is obviously fixed by every iterate of $h$ as well. The question
is then about the existence of some $\overline{h}$ in the pseudogroup fixing $p$ and not (locally) coinciding with
an iterate of $h$. In other words, the question is whether or not every point different from $0 \in \C$ has cyclic
(possibly trivial) stabilizer.

The answers to the above mentioned problems given in this work sit naturally in the continuation of
\cite{MRR} as will be seen below. Some terminology is however needed before stating our main results.
First, given an element $f \in \diff$, we shall denote by $f^j$ its $j^{\rm th}$-iterate for $j \in \Z$
($f^0 = {\rm id}$ and $f^j = (f^{-1})^{\vert j\vert}$ for $j <0$). This definition has a clear meaning in
terms of germs whereas this meaning is less clear in the context of pseudogroups, cf. Section~2 for details.

Now consider $\diff$ equipped with the {\it analytic topology}\, (cf. Section~2) turning $\diff$ into a Baire
space. Denote by $\diffalpha$ the normal subgroup of $\diff$ consisting of those germs of diffeomorphisms
that are tangent to the identity to order $\alpha \in \N$ (if $\alpha =0$ then
$\diffalpha =\diff$). The subgroup $\diffalpha \subseteq \diff$ is closed for the analytic
topology and, in fact, it is a Baire space in its own right.
Next, let $\diffalpha \times \diffalpha$ be endowed with the product
analytic topology. Suppose we are given
two elements $f,g$ in $\diff$ and denote by $G_1$ (resp. $G_2$) the cyclic group generated
by $f$ (resp. $g$). Naturally, the groups $G_1, G_2$ may or may not be finite and their orders are respectively
the orders of the germs $f, g$ which are denoted by $\textsf{r}$ and $\textsf{s}$. In other words,
$\textsf{r}$ (resp. $\textsf{s}$) is the smallest strictly positive integer for which $f^{\textsf{r}} ={\rm id}$
(resp. $g^{\textsf{s}} ={\rm id}$).
If this integer does not exist, then we set $\textsf{r} =\infty$ (resp. $\textsf{s} = \infty$) and,
in this case, the group $G_1$ (resp. $G_2$) turns out to be infinite and isomorphic to $\Z$.

Now denote by ${\rm F}_2$ the free group on two generators $a,b$ and consider the natural evaluation morphism
from ${\rm F}_2$ to $\diff$ consisting of making the substitutions $a\mapsto f$ and $b \mapsto g$
(and of interpreting the ``concatenation of letters'' as composition of germs).
Let ${\rm N}$ be the {\it normal subgroup}\, of ${\rm F}_2$ generated by $\{ a^{\textsf{r}} ,
b^{\textsf{s}} \}$, with the convention that $a^{\infty} = b^{\infty}= {\rm id}$. The
quotient group ${\rm F}_2 /{\rm N}$ is isomorphic to the {\it free product}\, $G_1 \ast G_2$
of the groups $G_1, G_2$. Furthermore, the  above mentioned evaluation morphism
factors through the quotient ${\rm F}_2 /{\rm N}$ so as to induce a homomorphism $\mathcal{E}$
from $G_1 \ast G_2$ to $\diff$.

Another explicit construction for the homomorphism $\mathcal{E}$ consists of using the fact
that every element in the free product $G_1 \ast G_2$ is represented by a
{\it unique}\, reduced word in the letters $a, b$, where the empty-word
represents the identity (cf. Section~2 for further details). Therefore, the elements
of $G_1 \ast G_2$ are naturally identified to reduced words $W (a, b)$. With this notation,
$\mathcal{E} (W (a, b))$ is simply the element
of $\diff$ obtained by substituting $a \mapsto f$ and $b \mapsto g$ in the spelling of
$W (a, b)$ (where again the ``concatenation of letters'' becomes composition of
germs). In the sequel, the element of $\diff$ given by $\mathcal{E} (W (a, b))$
is going to be denoted by $W (f, g)$.

To state Theorem~A, recall that a local diffeomorphism $f$ fixing $0 \in \C$ is linearizable
if and only if it is conjugate to the linear map $z \mapsto f'(0) \, z$ by a local holomorphic change of coordinates,
where $f'(0)$ stands for the derivative of $f$ at $0 \in \C$. This local diffeomorphism is said to have
a {\it Cremer point}\, (at $0 \in \C$) if it is not linearizable and if $f'(0)$ has norm~$1$
but is not a root of unity. Suppose now that we are given local diffeomorphisms $f,g$ such that none of
them has a Cremer point at $0 \in \C$. If $W(a,b)$ is a non-empty reduced word in $a,b$ (w.r.t
$G_1 \ast G_2$), the following was shown in \cite{MRR} (cf. Theorem~\ref{teo_MRR}):
there is a $G_{\delta}$-dense $\mathcal{V} \subset \diffalpha \times \diffalpha$ such that, whenever $(h_1, h_2)
\in \mathcal{U}$, the element $W (h_1^{-1} \circ f \circ h_1, h_2^{-1} \circ g \circ h_2)$ of the pseudogroup generated by
$h_1^{-1} \circ f \circ h_1, h_2^{-1} \circ g \circ h_2$ does not coincide with the identity on any
connected component of its domain of definition. Equivalently, all fixed points of
$W (h_1^{-1} \circ f \circ h_1, h_2^{-1} \circ g \circ h_2)$ are isolated, though not necessarily hyperbolic.
Here we shall improve on this result by proving Theorem~A below.

In the sequel $\alpha \in \N$ is fixed as well as local diffeomorphisms $f,g$. Given local diffeomorphisms
$h_1, h_2$, we denote by $\Gamma_{h_1,h_2}$ the pseudogroup generated by
$h_1^{-1} \circ f \circ h_1, h_2^{-1} \circ g \circ h_2$ on some local neighborhood of $0 \in \C$ (to be chosen
later).

\vspace{0.1cm}

\noindent {\bf Theorem A}. {\sl
Suppose we are given $f,g$ in $\diffalpha$ and denote by $D$ an open disc about $0 \in \C$ where $f, g$ and their inverses are
defined. Assume that none of the local diffeomorphisms $f, g$ has a Cremer point at $0 \in \C$.
Then, there is a $G_{\delta}$-dense set $\mathcal{U} \subset \diffalpha \times \diffalpha$ such that,
whenever $(h_1, h_2)$ lies in $\mathcal{U}$, the
pseudogroup $\Gamma_{h_1,h_2}$ generated by $\tilde{f} = h_1^{-1} \circ f \circ h_1, \, \tilde{g} =
h_2^{-1} \circ g \circ h_2$ on $D$ satisfies the following:
\begin{enumerate}
\item The stabilizer of every point $p \in D$ is either trivial or cyclic.

\item There is a sequence of points $\{ Q_i\}$, $Q_i \neq 0$ for every $i \in \N^{\ast}$, converging to $0 \in \C$
and such that every $Q_n$ is a hyperbolic fixed point of some element $W_i (\tilde{f}, \tilde{g}) \in \Gamma_{h_1,h_2}$.
Furthermore the orbits under $\Gamma_{h_1,h_2}$ of $Q_{n_1}, \, Q_{n_2}$ are disjoint provided that $n_1 \neq n_2$.
\end{enumerate}
}

\vspace{0.1cm}

Note that, in the statement of Theorem~A, the analytic conjugacy classes of $f$ and $g$ in $\diff$
are supposed to be fixed. This condition is naturally imposed by the use of the Krull topology in classical problems
about singular foliations, cf. below.

The assumption that neither $f$ nor $g$ has a Cremer point at $0 \in \C$
is indeed necessary for the statement of Theorem~A to hold, cf. Section~2. On the other hand,
the theorem holds equally well for pseudogroups generated by every finite collection of local diffeomorphisms.
Some comments on possible improvements of Theorem~A also deserve to be included here. Essentially these improvements concern
extensions ``generic/general'' for the corresponding statement. As to item~(2), it is conceivable that the existence of ``more than
one'' orbit of hyperbolic fixed points may be verified for every non-solvable pseudogroup of $\diff$. On the other hand,
item~(1) is unlikely to be a general phenomenon and we believe that only a ``generic'' affirmative answer can be
expected. A more optimistic perspective might suggest that an affirmative answer can still hold true for ``open and
dense sets'' though this already seems a bit unlikely to happen. In any event, in order to make further progress in
this type of questions, it seems clear that the next step is to investigate the possible existence of ``Closing lemmas''
for the mentioned pseudogroups.

Theorem~A has a few consequences on the topology of the leaves of a foliation on
a neighborhood of an invariant curve and/or a neighborhood of the singular point since in most well-known
problems this information can be codified into the holonomy pseudogroup associated to the invariant curve
(or to the reduction of the singular point). An important class of applications stems from classical problems
abut deciding the topology of leaves for a (local) ``generic'' foliation, where ``generic'' is understood in terms
of dense sets for the Krull topology (see \cite{lefloch}, \cite{marinmattei}, \cite{MRR}, \cite{matteisalem}). These
applications are by now standard so that they will not be detailed in this paper (apart from stating Corollary~B).
On the other hand, for the convenience of non-experts, let us explain why the use of Krull topology forces us to
fix the conjugacy classes of the initial local diffeomorphisms. For this, suppose that $\fol$
is the singular foliation associated to the local orbits of a holomorphic vector field $X$ defined on a neighborhood
of $(0,0) \in \C^2$. Let $\widetilde{\fol}$ denote the transform of $\fol$ by a suitable birational transformation so that
$\widetilde{\fol}$ is defined on a neighborhood of an invariant divisor. Suppose that the singularities
of $\widetilde{\fol}$ are all {\it hyperbolic}\, (an assumption frequently satisfied). Note that, in the present
context, a singular point $p$ of $\widetilde{\fol}$ is said to be hyperbolic if $\widetilde{\fol}$ can locally be given
by a holomorphic vector field $Y$ whose linear part at $p$ has non-zero eigenvalues $\lambda_1, \lambda_2$ verifying
$\lambda_1/\lambda_2 \in \C \setminus \R$. A classical problem in differential equations/singularity theory considers
dynamical/topological properties of foliations $\fol$ as before that are satisfied by Krull-dense sets of foliations.
In practice, this means that we should look for properties that are verified by foliations $\fol'$ given by vector fields
$X'$ having the same Taylor series as $X$ up to an arbitrarily fixed order. Since the invariant divisor of
$\widetilde{\fol}$ as well as the position of its singular points and the values of corresponding eigenvalues are determined
by some finite jet of $X$, these data cannot be changed in the construction of the ``perturbed'' foliation. Thus
the holonomy maps arising from corresponding singular points of $\widetilde{\fol}$ and of the ``perturbed foliation''
are necessarily conjugate to each other since they are both linearizable in view of Poincar\'e theorem. In other words,
while constructing foliations ``near'' to $\fol, \widetilde{\fol}$, the analytic conjugacy classes of the mentioned
holonomy maps are necessarily fixed.

In general, the transverse structure of a (singular) holomorphic foliation is described through the pseudogroup
associated to the holonomy with respect to an invariant curve. The existence of hyperbolic fixed points, for example,
is strictly related to the presence of hyperbolic limit cycles for the foliation in question. Though we have
mentioned hyperbolic singularities to illustrate the interest of having fixed analytic conjugacy classes in the
statement of Theorem~A, we shall explicitly state an application of this theorem in a rather ``orthogonal'' setting
where singularities are far from hyperbolic. This setting corresponds to the much studied case of nilpotent
foliations associated to Arnold $A^{2n+1}$ singularities.

Recall that a {\it nilpotent foliation}\, about $(0,0) \in \C^2$ is the singular foliation associated to the local
orbits of a (local) holomorphic vector field $X$ having an isolated singularity at $(0,0) \in \C^2$ where the linear
part of $X$ at $(0,0)$ is nilpotent (different from zero). Consider then nilpotent foliations
possessing a unique separatrix which is given by a cusp of the form $\{y^2 + x^{2n+1} = 0\}$, i.e.
the separatrix is an analytic curve locally equivalent to the mentioned cusp.
Foliations in this class are called nilpotent foliation $\fol$ of type $A^{2n+1}$.
In \cite{MRR}, it was proved that a generic nilpotent foliation $\fol$ of type $A^{2n+1}$ possesses
only countably many non-simply connected leaves.

Let $X \in \mathfrak{X}_{(\C^2,0)}$ be a holomorphic vector field with an isolated singularity at the origin and defining
a germ of nilpotent foliation $\fol$ of type $A^{2n+1}$.
By relying on the construction carried out in Section~5 of \cite{MRR}, Theorem~A yields:

\vspace{0.1cm}

\noindent {\bf Corollary B (Cusps)}. {\sl For every (sufficiently large) $N \in \N$, there exists a vector field
$X^{\prime} \in \mathfrak{X}_{(\C^2,0)}$ defining a germ of a foliation $\fol^{\prime}$ and satisfying the
following conditions:
\begin{itemize}
\item[(1)] $J_0^N X^{\prime} = J_0^N X$ (i.e. the vector field $X, \, X^{\prime}$ are tangent to order $N$ at the origin).

\item[(2)] The foliations $\fol$ and $\fol^{\prime}$ have $S$ as a common separatrix.

\item[(3)] There exists a fundamental system of open neighborhoods $\{U_n\}_{n \in \N}$ of $S$, inside a
closed ball $\bar{B}(0,R)$, such that for all $n \in \N$, the leaves of the restriction of $\fol^{\prime}$
to $U_n \setminus S$ are simply connected except for a countable set of them.

\item[(4)] The countable set mentioned above is, in fact, infinite.

\item[(5)] Every leaf of the restriction of $\fol^{\prime}$ to $U_n \setminus S$ is either simply-connected or
topologically equivalent to a cylinder.

\end{itemize}
}

Compared to the analogous statement in \cite{MRR}, the improvements made in this paper lies in items~(4) and~(5).
The remaining items make up the previous result in \cite{MRR} and depend solely on the fact that,
for a ``generic choice'' of $(h_1,h_2) \in \diffalpha \times \diffalpha$, every element in
the resulting pseudogroup $\Gamma_{h_1,h_2}$ coinciding with the identity on a
non-empty open set must coincide with the identity on all of its domain of definition. Naturally items~(4)
and~(5) are implied by our Theorem~A.

The proof of Theorem~A naturally starts from the above mentioned result for the pseudogroup $\Gamma_{h_1,h_2}$.
However it requires a rather different type of analysis which is vaguely reminiscent from the classical Kupka-Smale
theorems or, more precisely, with the part of its statement asserting that ``generic''
diffeomorphisms have only hyperbolic periodic points, cf. for example \cite{katok}. In closing this
Introduction, let us give a brief outline of
the structure of the proof of Theorem~A. Compared to Kupka-Smale statement, the first difference is naturally the fact
that we are dealing with a pseudogroup of local diffeomorphisms rather than with a single globally defined diffeomorphism.
This requires us to pay special attention to domains of definitions as well as to the boundary behavior of maps.
In Section~2, an appropriate setting to handle these pseudogroups will be worked out in detail.

The main analogy with Kupka-Smale theorem appears at the level of ``stability'' of hyperbolic fixed points. The idea
of ``stability'' for hyperbolic fixed points is materialized by saying that once a (local) diffeomorphism
possessing a (unique) hyperbolic fixed point is perturbed, the new diffeomorphism will also possess
a (unique) hyperbolic fixed point which, in addition, is ``near'' the initial fixed point. In our holomorphic
context, there is no need to worry about fixed points being hyperbolic since the desired ``stability'' property
is verified by ``multiplicity one'' fixed points of (local) holomorphic diffeomorphisms. In fact,
this type of ``stability'' follows at once from the
Argument principle (or Rouch\'e theorem). It is this application of the Argument principle
that, ultimately, will allow us to keep track of the number of fixed points that may arise when a diffeomorphism
having only isolated fixed points is perturbed. This material is carefully developed in Section~3 which also contains
the proof of Theorem~A modulo Proposition~\ref{disjointfixedpoints} whose proof, at this point, will essentially be reduced to
the construction of certain types of perturbations for pairs of local diffeomorphisms $(h_1,h_2) \in \diffalpha
\times \diffalpha$. These constructions will then be detailed in Sections~4 and~5. In particular, Section~4
begins with an outline of the structure of the proof of Proposition~\ref{disjointfixedpoints} itself.
As it is to be expected, our
``perturbations'' have no analogue in the Kupka-Smale context, not only due to the holomorphic nature of our
problem, but mainly due to the fact that the conjugacy classes for the initial diffeomorphisms $f,g$ have to be
fixed during all the procedure.

\vspace{0.1cm}

\noindent {\bf Acknowledgments}. We are very grateful to J.-F. Mattei for several discussions about the content
of this paper as well as for having encouraged us to work on this problem.

The authors thank the IMPA and the UMI-CNRS for their hospitality during the preparation of this paper. The first author
also thanks CNPq-Brazil for financial support. The second author thanks FCT for financial support through
CMUP and through the project PTDC/MAT/103319/2008.

\section{General set up}

In the sequel, let $\diff$ stand for the group of local holomorphic diffeomorphisms fixing $0 \in \C$ whereas
$\diffalpha$ stands for its normal subgroup consisting of elements tangent to the identity to order~$\alpha$
($\alpha =0$ corresponds to $\diff$ itself). Also, let $\Hol$ denote
the space of (germs of) holomorphic functions defined about $0 \in \C$. Clearly $\diff \subset \Hol$ and an
element $f \in \Hol$ belongs to $\diff$ if and only if $f'(0) \neq 0$.

Let us equip both $\diff$ and $\Hol$ with the so-called analytic topology (or $C^{\omega}$-topology)
that was first considered by Takens \cite{takens} in the context of real analytic diffeomorphisms of an analytic manifold who
also observed that it inherits of the Baire property. The definition of the analytic topology
can naturally be adapted to $\Hol$ and it was shown in \cite{MRR} that it turns $\Hol$ into a complete metric space.
The metric is defined as follows. Suppose that  $f, \, g$
in $\Hol$ are given and consider the holomorphic function $f-g$ which is defined in a neighborhood of
$0 \in \C$. Denote by $c_1 x + c_2 x^2 + \cdots$ the Taylor series of $f-g$ at $0 \in \C$. The metric
$d_A$ inducing the analytic topology in $\Hol$ is then given by
\[
d_A (f,g) = \sup_{k \in \N} \Vert c_k \Vert^{1/k} \, .
\]
Being a complete metric space, $\Hol$ has the Baire property. Since $\diff$ is clearly an open and dense
subset of $\Hol$, we recover the fact that $\diff$ has the Baire property as well. Furthermore, if $\alpha \geq 1$,
$\diffalpha$ is a complete metric space in its own right and hence it also possesses the Baire property.

Let $f, \, g$ be given elements in $\diff$ and denote by $G_1, \, G_2$ the cyclic group generated by $f, \, g$,
respectively. Fix a sufficiently small {\it open}\, disc $D$ about $0 \in \C$ where $f, \, g$ and their inverses are
defined. To begin the discussion concerning Theorem~A a few notions need to be recalled. Also, while
in the statement of Theorem~A and in the discussion below, we shall restrict our attention
to pseudogroups generated by two local diffeomorphisms $f,g$, all the arguments immediately carry over to
pseudogroups generated by a finite number of local diffeomorphisms.

Let us begin by recalling the formal notion of a reduced word in two letters. Let $f, \, g \in \diff$ be two
holomorphic diffeomorphisms fixing the origin of $\C$ and assume that they are both different from the identity. Denote
by $\textsf{r}$ (resp. $\textsf{s}$) the order of $f$  (resp. $g$), namely $\textsf{r} \in \N^{\ast}$ (resp. $\textsf{s}
\in \N^{\ast}$) is the smallest strictly positive integer for which $f^\textsf{r} = {\rm id} \in \diff$ (resp.
$g^\textsf{s} = {\rm id} \in \diff$). If $\textsf{r}$ (resp. $\textsf{s}$) does not exist, then the order
of $f$  (resp. $g$) is said to be~$\infty$. We shall write $\textsf{r} = \infty$ (resp. $\textsf{s} =
\infty$) to refer to the latter case and $\textsf{r} < \infty$ (resp. $\textsf{s} < \infty$) to indicate
the former one. If $\textsf{r}$ (resp. $\textsf{s}$) equals~$\infty$, then, by convention, $\Z/\textsf{r}
\Z$ (resp. $\Z /\textsf{s} \Z$) is isomorphic to $\Z$. In terms of the mentioned presentation, a {\it reduced
word}\, in the letters $a,b$ (sometimes also said in the letters $a,a^{-1}, b, b^{-1}$) is a word $W (a,b)$
whose spelling has the form $\vartheta_l^{r_l} \ast \cdots \ast \vartheta_1^{r_1}$ with the following rules
being respected:
\begin{enumerate}
  \item $\vartheta_i$ takes on the values $\{ a, b\}$.
  \item If $\vartheta_{i_0}$ takes on the value $a$ (resp. $b$) then $\vartheta_{i_0-1}$ and
  $\vartheta_{i_0 +1}$ take on the value $b$ (resp. $a$) provided that $\vartheta_{i_0-1}$ and
  $\vartheta_{i_0 +1}$ are defined.

  \item If $\vartheta_i$ takes on the value $a$, then $r_i$ takes values in the set
  $\{ 1, \ldots ,
  \textsf{r} -1 \}$ provided that $\textsf{r} < \infty$. If $\textsf{r} = \infty$, then $r_i$ takes values
  in $\Z^{\ast}$ (it is understood that, for $r_i < 0$, $a^{r_i}$ means $(a^{-1})^{\vert r_i
  \vert}$).

  \item Similarly, if $\vartheta_i$ takes on the value $b$, then $r_i$ takes values in the set $\{ 1, \ldots ,
  \textsf{s} -1 \}$ provided that $\textsf{s} < \infty$. If $\textsf{s} = \infty$, then $r_i$ takes values
  in $\Z^{\ast}$ (where $b^{r_i}$ means $(b^{-1})^{\vert r_i \vert}$ whenever
  $r_i < 0$).
\end{enumerate}

With the previous notations, let us consider the free product $G_1 \ast G_2$
between $G_1$ and $G_2$. In terms of presentation, this group is isomorphic
to the group defined by $\{ a, b \, ; \, a^{\textsf{r}} = b^{\textsf{s}} = {\rm id} \}$, where the relation
$a^{\textsf{r}} = {\rm id}$ (resp. $b^{\textsf{s}} = {\rm id}$) is understood to be void if $\textsf{r}
= \infty$ (resp. $\textsf{s} = \infty$). In the sequel we shall use this presentation to refer to the
free product $G_1 \ast G_2$.
It should be noticed that every element in $\{ a, b \, ; \,
a^{\textsf{r}} = b^{\textsf{s}} = {\rm id} \}$ is represented by a unique reduced word $W(a,b)$ (where
the neutral element corresponds to the empty word).

The fundamental object involved in the subsequent discussion is the notion of {\it pseudogroup} generated by local
diffeomorphisms $f$ and $g$ about $0 \in \C$. For local diffeomorphisms $f,g$ as above, consider an
open disc $D$ about $0\in \C$ where $f,f^{-1},
g, g^{-1}$ are all defined and one-to-one. We want to consider the {\it pseudogroup
$\Gamma = \Gamma (f,g,D)$ generated by $f,f^{-1}, g, g^{-1}$ on $D$} (in the sequel this pseudogroup will be referred
to as being generated by $f,g$ and their inverses or simply by $f,g$, when no confusion is possible).
An accurate definition of this pseudogroup is required for the subsequent discussion. Consider a fixed word $W (a,b) =
\vartheta_l^{r_l} \ast \cdots \ast \vartheta_1^{r_1}$. In the sequel, unless otherwise
stated, every word is supposed to be reduced with respect to the group
$\{ a, b \, ; \, a^{\textsf{r}} = b^{\textsf{s}} = {\rm id} \}$, with the
previously defined conventions about $\textsf{r}, \, \textsf{s}$. In some cases, we shall also consider
the spelling of $W (a,b)$ arising from splitting the components $\vartheta_i^{r_i}$. More precisely, when both
$\textsf{r}$ or $\textsf{s}$ equals $\infty$, since the exponents $r_j$ can be negative, the mentioned
splitting of the components $\vartheta_i^{r_i}$ takes on the form $\theta_s \ast \cdots \ast \theta_1$, where:
\begin{itemize}
  \item $\theta_j$ takes on one of the values $a,b, a^{-1}, b^{-1}$
  \item if $\theta_j$ takes on the value $a$ (resp. $a^{-1}$) then, whenever defined, neither $\theta_{j-1}$ nor
  $\theta_{j+1}$ takes on the value~$a^{-1}$ (resp. $a$). A similar rule applies to $b, b^{-1}$.
\end{itemize}
In the cases where both $\textsf{r}, \, \textsf{s} < \infty$, $\theta_j$ takes only on the values $a,b$ and every sequence
$\theta_i, \theta_{i+1},\ldots$ of ``$\theta_i$'' with the same value is contained in the split of some $\vartheta_j^{r_j}$
in the natural sense. Adaptations to the mixed cases $\textsf{r} < \infty, \, \textsf{s}= \infty$ or $\textsf{r} = \infty,
\, \textsf{s}< \infty$ are straightforward and left to the reader. In any event, we obtain $s = \sum_{i=1}^l \vert r_i \vert$.
Now, consider the corresponding local diffeomorphism $W(f,g)$ written under the form $H_{s} \circ \cdots \circ H_{1}$ where each
$H_{i}$, $i \in \{1, \ldots ,s\}$, belongs to the set $\{ f^{\pm 1},g^{\pm 1}\}$. In other words, $H_i$ replaces $\theta_i$ by
means of the substitutions $f^{\pm 1} \mapsto a^{\pm 1}$, $g^{\pm 1} \mapsto b^{\pm 1}$. The {\it domain of definition}\, of
$W(f,g) =H_{s} \circ \cdots \circ H_{1}$ as an element of $\Gamma$ can be introduced by recursively defining the domains of
definitions of each element $H_{i} \circ \cdots \circ H_{1}$ of $\Gamma$, $i=1, \ldots ,s$, as follows:
\begin{itemize}
\item The domain of definition of $H_1$ is all of $D$ and, for every $z \in D$, $H_1 (z)$ is defined in the obvious way.

\item Suppose that the domain of definition ${\rm Dom}_{H_{i} \circ \cdots \circ H_{1}}$
of $H_{i} \circ \cdots \circ H_{1}$ is already known and that the image $H_{i} \circ \cdots \circ H_{1} (z)$
of $z \in {\rm Dom}_{H_{i} \circ \cdots \circ H_{1}}$ under $H_{i} \circ \cdots \circ H_{1}$ is also defined.
Then the domain of definition
${\rm Dom}_{H_{i+1} \circ \cdots \circ H_{1}}$ of $H_{i+1} \circ \cdots \circ H_{1}$ is obtained by setting
$$
{\rm Dom}_{H_{i+1} \circ \cdots \circ H_{1}} = \{ z \in {\rm Dom}_{H_{i} \circ \cdots \circ H_{1}} \; \; ; \; \;  H_{i} \circ \cdots \circ H_{1} (z) \in D \} \, .
$$
\end{itemize}
In particular, ${\rm Dom}_{H_{i+1} \circ \cdots \circ H_{1}} \subseteq {\rm Dom}_{H_{i} \circ \cdots \circ H_{1}}$ and hence
the domain of definition of every element in $\Gamma$ is naturally contained in $D$. Besides, for $z \in {\rm Dom}_{H_{i+1}
\circ \cdots \circ H_{1}}$, the value of $H_{i+1} \circ \cdots \circ H_{1} (z)$ is defined by setting $H_{i+1} \circ  \cdots
\circ H_{1} (z) = H_{i+1} \circ [H_{i} \circ \cdots \circ H_{1}] (z)$.

Let $f,g$ be local diffeomorphisms as above and consider now a disc $D$ whose closure $\overline{D}$ is contained in
a larger (open) disc where $f,g$ and their inverses are well-defined one-to-one maps. This disc $D$ will be fixed
in the sequel and its boundary is going to be denoted by $\partial D$.
With the previously defined notations, suppose we are given a word $W (a,b) =
\vartheta_l^{r_l} \ast \cdots \ast \vartheta_1^{r_1} = \theta_s \ast \cdots \ast \theta_1$ which, as always, is supposed
to be non-empty and reduced w.r.t. the group $\{ a, b \, ; \, a^{\textsf{r}} = b^{\textsf{s}} = {\rm id} \}$.
Given local diffeomorphisms $h_1, h_2 \in \diffalpha$, denote by $\tilde{f}, \, \tilde{g}$ the local diffeomorphisms
given by $\tilde{f} =h_1^{-1} \circ f \circ h_1$ and $\tilde{g} = h_2^{-1} \circ g \circ h_2$. In terms of pseudogroups
generated by $f,g$ and by $\tilde{f}, \tilde{g}$, note that the domains of definition of the corresponding elements
$W(f,g)$ and $W(\tilde{f}, \tilde{g})$ may drastically differ. To make sense of all these pseudogroups on the fixed disc
$D$, or on the closed disc $\overline{D}$, we need the following definition.

\begin{defnc}
\label{severaldomains}
\begin{enumerate}
\item The domain of definition of $h_1$ is defined as follows: let $\rho$ be the radius of the maximal open disc about
$0 \in \C$ in which $h_1$ is defined and one-to-one. Then the open domain of definition of $h_1$ is defined to be the
open disc of radius $9\rho/10$. The closed domain of definition of $h_1$ will also be considered and this will be nothing
but the closed disc of radius $9\rho/10$. Analogous definitions apply to each of the local diffeomorphisms: $h_1^{-1}, \,
h_2, \, h_2^{-1}$.

\item The domain of definition of $\tilde{f}=h_1^{-1} \circ f \circ h_1$ consists of those points $p$ verifying all the
following conditions: $p$ belongs to the open domain of definition of $h_1$, $h_1(p)$ belongs to the domain of definition
of $f$, i.e. to $D$ . Besides $f\circ h_1(p)$ must belong to the open domain of definition of $h_1^{-1}$. Analogous
considerations apply to the domain of definition of $\tilde{g}=h_2^{-1} \circ g \circ h_2$ and to $\tilde{f}^{-1},
\, \tilde{g}^{-1}$.

\item Finally, considering the pseudogroup generated by $\tilde{f}, \, \tilde{g}$ on the open disc $D$, the domain
of definition of the element $W (h_1^{-1} \circ f \circ h_1, \, h_2^{-1} \circ g \circ h_2) = W (\tilde{f}, \, \tilde{g})$
is obtained according to the above given general definitions concerning pseudogroups, with $\tilde{f}, \, \tilde{g}$ in the
place of $f, \, g$.

\item The (closed) domain of definition of $\tilde{f}=h_1^{-1} \circ f \circ h_1,\,
\tilde{g}=h_1^{-1} \circ g \circ h_1$ on $\overline{D}$ is defined according to item~(2) by using closed domains
of definition for $f,g$ as well as for $h_1, h_1^{-1}, h_2,h_2^{-1}$. Finally the pseudogroup generated by
$\tilde{f}, \tilde{g}$ on the closed disc $\overline{D}$ is such that the domain of definition of a general element
$W (h_1^{-1} \circ f \circ h_1, \, h_2^{-1} \circ g \circ h_2) = W (\tilde{f}, \, \tilde{g})$ is obtained by
starting with closed domains of definition for $\tilde{f}, \tilde{g}$ and following the general pseudogroup rules.
\end{enumerate}
\end{defnc}

Once $(h_1, h_2) \in \diffalpha \times \diffalpha$ is chosen, the pseudogroup generated by
$\tilde{f}=h_1^{-1} \circ f \circ h_1$ and by $\tilde{g}=h_1^{-1} \circ g \circ h_1$ on the open disc $D$
is going to be denoted by $\Gamma_{h_1,h_2}$. Consider now an element $W (\tilde{f}, \, \tilde{g})$
of $\Gamma_{h_1,h_2}$ and denote by ${\rm Dom}_W (D)$ its domain of definition. The local diffeomorphism
$W (\tilde{f}, \, \tilde{g})$ may also be considered as an element of the pseudogroup generated by
$\tilde{f}, \tilde{g}$ on the closed disc $\overline{D}$ and, in this case, it has a new domain of definition
${\rm Dom}_W (\overline{D})$. Clearly the topological closure $\overline{{\rm Dom}_W (D)}$ of ${\rm Dom}_W (D)$
is contained in ${\rm Dom}_W (\overline{D})$ whereas, in principle, these two sets may be distinct.
However, in what follows, we shall primarily be interested in the pseudogroup
$\Gamma_{h_1,h_2}$ and, when studying its elements, it will be important to consider their local behaviors around
points in $\overline{{\rm Dom}_W (D)}$. Thus possible points
lying in ${\rm Dom}_W (\overline{D}) \setminus \overline{{\rm Dom}_W (D)}$ will play no role in the subsequent
discussion. In view of this and to abridge notations, by a small abuse of language we shall
consider that elements in the pseudogroup generated by
$\tilde{f}=h_1^{-1} \circ f \circ h_1$ and by $\tilde{g}=h_1^{-1} \circ g \circ h_1$ on the closed disc
$\overline{D}$ are defined only on $\overline{{\rm Dom}_W (D)}$. In other words, we shall write
${\rm Dom}_W (\overline{D})$ to denote the domain of definition in question and, throughout the paper,
${\rm Dom}_W (\overline{D})$ will be assumed to coincide with the closure of ${\rm Dom}_W (D)$.

The pseudogroup generated by $\tilde{f}, \tilde{g}$
on the (closed) disc $\overline{D}$ will be denoted by $\overline{\Gamma}_{h_1,h_2}$.
Since $0 \in \C$ is fixed by $f,g$, we conclude that every word has a non-empty domain of definition as element of
both $\Gamma_{h_1,h_2}$ and $\overline{\Gamma}_{h_1,h_2}$.
Furthermore, since non-constant holomorphic maps are open maps,
the domain of definition of every element $W (\tilde{f}, \, \tilde{g})$ in
$\Gamma_{h_1,h_2}$ is necessarily an open set.
It may, however, be {\it disconnected}.

The above given definition of $\overline{\Gamma}_{h_1,h_2}$ involves a technical issue playing a significant
role in the forthcoming sections which is singled out by the following lemma.

\begin{lema}
\label{extensionoverclosure}
Let $W (\tilde{f}, \, \tilde{g})$ be an element in $\overline{\Gamma}_{h_1,h_2}$ whose domain of definition
is denoted by ${\rm Dom}_W (\overline{D})$. Suppose that $U$ is a connected component of ${\rm Dom}_W (\overline{D})$.
Then $U$ is a closed set and, moreover, the holomorphic map $W (\tilde{f}, \, \tilde{g}) : U \rightarrow \C$
possesses a holomorphic extension to some open neighborhood of $U$.
\end{lema}

\begin{proof}
The fact that $U$ is closed follows again from the fact that the maps $\tilde{f}, \tilde{g}$ are holomorphic
and hence open (since non-constant). The existence of the desired holomorphic extension follows from the fact
that $f,g$ have holomorphic extensions to a neighborhood of $\overline{D}$ along with the construction
of closed domains of definition for $h_1, h_1^{-1}, h_2, h_2^{-1}$ given in
Definition~\ref{severaldomains}. The lemma is proved.
\end{proof}

Considering the pseudogroups as above, there is already a point to be made about the iterates $f^j,\, g^j$,
of $f,\, g$, which should themselves be understood as elements of suitable
pseudogroups defined on some open neighborhood of $0 \in
\C$. More generally, given $F \in \diff$ and fixed a neighborhood $D$ of $0 \in \C$ where $F$ is defined, the notation
$F^j$, where $j \in \Z^{\ast}$, refers to the element $F^j$ viewed as an element of the pseudogroup generated by $F$ on $D$.
Now, by combining Poincar\'e theorem on the existence of linearizing coordinates about hyperbolic fixed points
to the well-known topological dynamics
of diffeomorphisms tangent to the identity (see for example \cite{carleson}), the following lemma immediately follows.

\begin{lema}
\label{nocremer}
Suppose that $F \in \diff$ is a local diffeomorphism that does not have a Cremer point at $0 \in \C$.
Then there is a small disc $B (\varepsilon)$ about $0 \in \C$ where, for every $j \in \Z^{\ast}$,
$F^j$ has no fixed point unless $F^j$ coincides with the
identity on all of its domain of definition. In particular,
if $F^j$, $j \in \Z$, coincides with the identity on some connected component of its domain of definition,
then it coincides with the identity on all of its domain of definition.\qed
\end{lema}

In closing this section, let us point out that Lemma~\ref{nocremer} is no longer valid if $F$ is allowed to
have a Cremer point at $0 \in \C$. Also, as it was the case in \cite{MRR}, the conclusions of our Theorem~A
no longer holds if $f,g$ are allowed to have Cremer points. This is a consequence of a construction due to
Perez-Marco, cf. \cite{perezmarco} of local diffeomorphisms $F \in  \diff$ exhibiting a Cremer point
at $0 \in \C$ and satisfying the following conditions:
\begin{itemize}

\item There exists a sequence
of points $\{q_i \}$ accumulating on $0 \in \C$ along with a sequence of {\it periods} $\{ n_i\}$, $n_i \neq 0$,
going to infinity such that $F^{n_i} (q_i) =q_i$ for every $i \in \N$.

\item The dynamics of $F^{n_i}$ about
its fixed point $q_i$ may arbitrarily be fixed: in particular, it can be chosen so that $F^{n_i}$ coincides with the
identity on some (very small) neighborhood of $q_i$.
\end{itemize}

\section{Proof of Theorem~A}

Let $f,g \in \diff$ and $D \subset \C$ be as in Section~2. Fixed $\alpha \in \N$ and given $(h_1,h_2) \in
\diffalpha \times \diffalpha$, consider the pseudogroups $\Gamma_{h_1,h_2}$ and $\overline{\Gamma}_{h_1,h_2}$
introduced in Section~2. Given a point $p \in D$, consider all those elements $W (\tilde{f}, \tilde{g})$ in
$\Gamma_{h_1,h_2}$ verifying the following:
\begin{itemize}
  \item The point $p$ belongs to the domain of definition of $W (\tilde{f}, \tilde{g})$ (as element of $\Gamma_{h_1,h_2}$).
  \item We have $W (\tilde{f}, \tilde{g}) (p) =p$.
\end{itemize}
The germs at $p$ of all elements $W (\tilde{f}, \tilde{g})$ in $\Gamma_{h_1,h_2}$ satisfying the mentioned
conditions form a group named {\it the stabilizer of $p$ in $\Gamma_{h_1,h_2}$}. The stabilizer of $p$ is said to
be trivial if this group is reduced to the identity.

As mentioned, Theorem~A is a natural continuation of the results obtained in \cite{MRR}. In fact, to begin the
approach to Theorem~A, we remind the reader that the main result of \cite{MRR} can be summarized as follows.

\begin{teo}\cite{MRR}\label{teo_MRR}
Suppose that $f,g$ and $D$ as above are fixed. Then there exists a $G_{\delta}$-dense set
$\mathcal{V} \subset \diffalpha \times \diffalpha$ such that, whenever $(h_1,h_2) \in \mathcal{V}$, for every non-empty
reduced word $W (a,b)$ the corresponding element $W (\tilde{f}, \tilde{g})$ of $\Gamma_{h_1,h_2}$ does
not coincide with the identity on any connected component of its domain of definition.
\end{teo}

In the rest of this section we shall prove Theorem~A modulo a more technical statement, namely
Proposition~\ref{disjointfixedpoints}, whose proof
will be deferred to Sections~4 and~5.

Let us begin with Proposition~\ref{herewego1} below which plays a significant role in everything that follows.
This proposition, as well as the argument used in its proof, is already important for the discussion conducted
in this section although its full strength will only be needed in the next section.

As always, let $W (a,b)$ be a fixed non-empty reduced word. Given $(h_1,h_2) \in \diffalpha \times \diffalpha$,
consider the pseudogroup $\Gamma_{h_1,h_2}$ (resp. $\overline{\Gamma}_{h_1,h_2}$) generated
by $\tilde{f} = h_1^{-1} \circ f \circ h_1$ and by $\tilde{g} = h_2^{-1} \circ g \circ h_2$ on $D$
(resp. $\overline{D}$). Finally
recall that both $f, g$ are supposed not to have Cremer points at $0 \in \C$.
The reader is also reminded that, according to
our conventions, the domain of definition ${\rm Dom}_W (\overline{D})$ of an element $W (\tilde{f}, \tilde{g})
\in \overline{\Gamma}_{h_1,h_2}$ is the closure of its domain of definition ${\rm Dom}_W (D)$ when viewed as
element of $\Gamma_{h_1,h_2}$.

\begin{prop}
\label{herewego1}
There is an open and dense set $\mathfrak{U}_W \subset \diffalpha \times \diffalpha$ such that, whenever $(h_1, h_2)
\in \mathfrak{U}_W$, the element $W (\tilde{f}, \tilde{g})$ of $\overline{\Gamma}_{h_1,h_2}$ has only finitely many
fixed points in its (closed) domain of definition ${\rm Dom}_W (\overline{D})$.
\end{prop}

\begin{proof}
Naturally the set $\mathfrak{U}_W$ is defined as consisting of those pairs $(h_1,h_2) \in \diffalpha \times \diffalpha$
for which the statement of our proposition holds. Let us begin by showing that $\mathfrak{U}_W$ is dense. For this, it suffices
to check that every pair $(h_1,h_2)$ lying in the $G_{\delta}$-dense set $\mathcal{V}$ provided by Theorem~\ref{teo_MRR}
gives rise to a pseudogroup $\overline{\Gamma}_{h_1,h_2}$ whose elements, apart from the identity,
have only finitely many fixed points in their (closed) domains of definition.

Consider then $W (\tilde{f}, \tilde{g})$ where $h_1, \, h_2$ are as above and denote by ${\rm Dom}_W (\overline{D})$
its (closed) domain of definition. By construction, ${\rm Dom}_W (\overline{D})$ is the closure of the domain
of definition ${\rm Dom}_W (D)$ of $W (\tilde{f}, \tilde{g})$ viewed as element of $\Gamma_{h_1,h_2}$. Suppose for a
contradiction that $W (\tilde{f}, \tilde{g})$ possesses infinitely many fixed points in ${\rm Dom}_W (\overline{D})$.
Since ${\rm Dom}_W (\overline{D})$ is closed, there is $P \in {\rm Dom}_W (\overline{D})$ that is the limit of
a non-trivial sequence $\{ p_i\} \subset {\rm Dom}_W (\overline{D})$ consisting of fixed points of $W (\tilde{f}, \tilde{g})$.
However, by virtue of Lemma~\ref{extensionoverclosure}, $W (\tilde{f}, \tilde{g})$
admits a holomorphic extension to some open neighborhood $V \subset \C$ of $P$. Besides $W (\tilde{f}, \tilde{g})$
must coincide with the identity on $V$ since $P$ is accumulated by a non-trivial sequence $\{ p_i \}$ of
fixed points of $W (\tilde{f}, \tilde{g})$. On the other hand, the open set $V$ must intersect the open set
${\rm Dom}_W (D)$ non-trivially since $P$ lies in the closure of ${\rm Dom}_W (D)$. Therefore
$W (\tilde{f}, \tilde{g})$ coincides with the identity on some connected component of ${\rm Dom}_W (D)$
what is impossible since $(h_1,h_2) \in \mathcal{V}$. The resulting contradiction ensures that
$\mathfrak{U}_W$ is dense in $\diffalpha \times \diffalpha$.

It remains to show that $\mathfrak{U}_W$ is also open in the analytic topology. For this consider an element
$(h_1, h_2) \in \mathfrak{U}_W$ and the corresponding element $W (\tilde{f}, \tilde{g})$
of $\overline{\Gamma}_{h_1,h_2}$. We need to check that every pair $(\overline{h}_1, \overline{h}_2) \in
\diffalpha \times \diffalpha$ sufficiently close to $(h_1,h_2)$ must belong to $\mathfrak{U}_W$ as well.

Since $(h_1,h_2)$ lies in $\mathfrak{U}_W$, $W (\tilde{f}, \tilde{g})$ has
only finitely many fixed points in its (closed) domain of definition ${\rm Dom}_W (\overline{D})$. These points
will be denoted by $p_1, \ldots ,p_n$. Now we consider a small open neighborhood $U^{\epsilon}$ of
${\rm Dom}_W (\overline{D})$ such that $W (\tilde{f}, \tilde{g})$ still possesses a holomorphic extension
to a neighborhood of the closure $\overline{U}^{\epsilon}$ of $U^{\epsilon}$, cf. Lemma~\ref{extensionoverclosure}.
Modulo choosing $U^{\epsilon}$ very small,
the points $p_1, \ldots ,p_n$ are isolated points also for the set of fixed points of $W (\tilde{f}, \tilde{g})$
on the closed set $\overline{U}^{\epsilon}$. In fact, we can suppose that the latter set of fixed points coincide
with $p_1, \ldots ,p_n$, i.e. $W (\tilde{f}, \tilde{g})$ has no additional fixed point in $\overline{U}^{\epsilon}$.
In particular none of these fixed points lies in the boundary of $\overline{U}^{\epsilon}$.

Next, for each $i=1, \ldots ,n$, let $B_i (\delta)$ denote the disc of radius $\delta >0$ about $p_i$. If $\delta$
is chosen sufficiently small, then $B_i (\delta) \subset U^{\epsilon}$ and $B_i (\delta) \cap
B_j (\delta) =\emptyset$ provided that $i \neq j$. Let now $K$ be the compact set given by
$K = \overline{U}^{\epsilon} \setminus \bigcup_{i=1}^n B_i (\delta)$. Since $K$ is compact and
$W (\tilde{f}, \tilde{g})$ has no fixed point in $K$, it follows the existence of $\tau >0$ such that
$\Vert W (\tilde{f}, \tilde{g}) (z) -z \Vert \geq \tau >0$ for every~$z \in K$.

Finally consider a sequence $\overline{h}_{1,j}$ (resp. $\overline{h}_{2,j}$) of elements in $\diffalpha$
converging to $h_1$ (resp. $h_2$) in the analytic topology. Set $\overline{f}_j = \overline{h}_{1,j}^{-1}
\circ f \circ \overline{h}_{1,j}$ and $\overline{g}_j = \overline{h}_{2,j}^{-1}\circ g \circ \overline{h}_{2,j}$
and consider the corresponding element $W (\overline{f}_j, \overline{g}_j)$ of the pseudogroup generated by
$\overline{f}_j, \, \overline{g}_j$ on $D$. The proposition is now reduced to the following claim:

\vspace{0.1cm}

\noindent {\it Claim}. For sufficiently large $j$, $W (\overline{f}_j, \overline{g}_j)$ has only finitely
many fixed points in its domain of definition.

\vspace{0.1cm}

\noindent {\it Proof of the Claim}. Since convergence in the analytic topology ensures both convergence
of domains of definition as well as uniform convergence in these domains, modulo choosing $j$ very large,
the local diffeomorphism $W (\overline{f}_j, \overline{g}_j)$ satisfies the conditions below.
\begin{enumerate}
  \item The domain of definition of $W (\overline{f}_j, \overline{g}_j)$ as element of the pseudogroup generated by
$\overline{f}_j, \, \overline{g}_j$ on $D$ is contained in $U^{\epsilon} \subset \overline{U}^{\epsilon}$.
Besides $W (\overline{f}_j, \overline{g}_j)$
possesses a holomorphic extension to a neighborhood of $\overline{U}^{\epsilon}$.

  \item For every $z \in K \subset \overline{U}^{\epsilon}$, we have
$\Vert W (\overline{f}, \overline{g}) (z) -z \Vert \geq \tau/2 >0$. In particular, $W (\overline{f}_j, \overline{g}_j)$
has no fixed point in $K$.
\end{enumerate}
To conclude the proof of the claim, it suffices to check that $W (\overline{f}_j, \overline{g}_j)$ can have only
finitely many fixed points in $B_i (\delta)$, for every $i=1, \ldots ,n$. This is however an immediate
consequence of the Argument principle. Indeed, the uniform convergence of $W (\overline{f}_j, \overline{g}_j)$
to $W (\tilde{f}, \tilde{g})$ on (a neighborhood of $B_i (\delta)$) implies, by Cauchy formula and the Argument
principle, that the sum of the zeros of $W (\overline{f}_j, \overline{g}_j) (z) -z$ on $B_i (\delta)$ counted with
multiplicity  equals the multiplicity of $p_i$ as zero of $W (\tilde{f}, \tilde{g}) (z) -z$ since the values of the
corresponding integrals must agree by uniform convergence of the integrand. The claim is proved.\qed

The proof of Proposition~\ref{herewego1} is over.
\end{proof}

For posterior reference, it is useful to explicitly state some by-products of the proof of
Proposition~\ref{herewego1}. First we have:

\begin{coro}
\label{corobyproduct}
Consider a non-empty reduced word $W (a,b)$. Suppose that $(h_1,h_2) \in \diffalpha \times \diffalpha$
is such that the element $W (h_1^{-1} \circ f \circ h_1, h_2^{-1} \circ g \circ h_2) = W (\tilde{f}, \tilde{g})$
has only finitely many fixed points in its (open) domain of definition ${\rm Dom}_W (D)$ when viewed as element of
$\Gamma_{h_1,h_2}$. Then $W (\tilde{f}, \tilde{g})$ also has finitely many fixed points in its (closed) domain
of definition ${\rm Dom}_W (\overline{D})$ when viewed as element of $\overline{\Gamma}_{h_1,h_2}$.
\end{coro}

\begin{proof}
Since $W (\tilde{f}, \tilde{g})$ has only finitely many fixed points in ${\rm Dom}_W (D)$, it follows that
$W (\tilde{f}, \tilde{g})$ does not coincide with the identity on any connected component of the open set
${\rm Dom}_W (D)$. Now the second paragraph in the proof of Proposition~\ref{herewego1} can be repeated word-by-word
to establish the desired statement.
\end{proof}

In the above discussion, it is understood that the fixed points in question can be counted with their multiplicities.
On the other hand, if $(h_1,h_2)$ leads to an element $W (\tilde{f}, \tilde{g})$ having only finitely many fixed points
in its (closed) domain of definition ${\rm Dom}_W (\overline{D})$, then the proof of
Proposition~\ref{herewego1} also establishes that every pair $(\overline{h}_1, \overline{h}_2) \in
\diffalpha \times \diffalpha$ sufficiently close to $(h_1,h_2)$ gives rise to a new pseudogroup
$\Gamma_{\overline{h}_1, \overline{h}_2}$ on $D$ whose corresponding element $W (\overline{f}, \overline{g})$
has only finitely many fixed points in its domain of definition. For reference, we state:

\begin{lema}
\label{forlaterreference1}
Suppose that a reduced word $W (a,b)$ is fixed. Let $\mathfrak{V}_W \subset \diffalpha \times \diffalpha$
denote the set of pairs $(h_1,h_2)$ giving rise to a pseudogroup $\Gamma_{h_1,h_2}$ for which the element
$W (\tilde{f}, \tilde{g})$ possesses only finitely many fixed points in its closed
domain of definition. Then $\mathfrak{V}_W$ is open.

Furthermore, if $(h_1,h_2) \in \mathfrak{V}$ is ``perturbed'' to $(\overline{h}_1, \overline{h}_2)$
(necessarily in $\mathfrak{V}$), then the fixed points of $W (\overline{h}_1, \overline{h}_2)$ are close to the
fixed points of $W (\tilde{f}, \tilde{g})$.\qed
\end{lema}

Note that, when it comes to the assumption of Lemma~\ref{forlaterreference1}, Corollary~\ref{corobyproduct} ensures
that the mentioned assumption is, indeed, equivalent to saying that $W (\tilde{f}, \tilde{g})$ has finitely many
fixed points in its open domain of definition ${\rm Dom}_W (D)$.

Two reduced words $W_1 (a,b), \, W_2 (a,b)$ are said to be
{\it commensurable}\, if there is a word $W_{1-2} (a,b)$ such that both $W_1 (a,b), \, W_2 (a,b)$
are obtained by concatenating finitely many
``copies'' of $W_{1-2} (a,b)$. This means that the local diffeomorphism $W_1 (f,g)$ (resp. $W_2 (f,g)$)
is a finite power of the local diffeomorphism $W_{1-2} (f,g)$ for every pair $f,g \in \diff$. If there is
not such word $W_{1-2} (a,b)$, then $W_1 (a,b), \, W_2 (a,b)$ are said to be {\it incommensurable}. Let
$\mathcal{I}$ denote the collection of pairs $(W_1 (a,b), W_2 (a,b))$ of incommensurable words.
Clearly $\mathcal{I}$ is a countable set.

Given a reduced word $W (a,b)$ and a pair $(h_1,h_2) \in \mathfrak{U}_W
\subset \diffalpha \times \diffalpha$, denote by
${\rm Fix}\, (W (h_{1}^{-1} \circ f \circ h_{1}, \, h_{2}^{-1} \circ g \circ h_{2}))$
the set of fixed points of $W (h_{1}^{-1} \circ f \circ h_{1}, \, h_{2}^{-1} \circ g \circ h_{2})
= W (\tilde{f}, \tilde{g})$ in its (closed) domain of definition ${\rm Dom}_W (\overline{D})$
(i.e. $W (\tilde{f}, \tilde{g})$ is viewed as element of $\overline{\Gamma}_{h_1,h_2}$).
Let us now choose two incommensurable words $W_i (a,b)$ and $W_j (a,b)$ so that the pair
$(W_i (a,b), W_j (a,b))$ defines an element of $\mathcal{I}$. Consider the set $\mathcal{U}_{i,j}
\subset \diffalpha \times \diffalpha$ consisting of those pairs $(h_1,h_2)$ of local diffeomorphisms
fulfilling the following two conditions:
\begin{itemize}
\item Both sets ${\rm Fix}\, (W_i (\tilde{f}, \tilde{g}))$
and ${\rm Fix}\, (W_j (\tilde{f}, \tilde{g}))$ are finite.

\item ${\rm Fix}\, (W_i (\tilde{f}, \tilde{g})) \cap
{\rm Fix}\, (W_j (\tilde{f}, \tilde{g}) ) = \{0\} \subset \C$.

\item The multiplicity of $0 \in \C$ as fixed point of $W_i (\tilde{f}, \tilde{g})$ and
of $W_j (\tilde{f}, \tilde{g})$ does not change by perturbation.

\end{itemize}
Note that the first item above means that $(h_1,h_2)$ lies in the intersection $\mathfrak{V}_i \cap
\mathfrak{V}_j$ which is open owing to Lemma~\ref{forlaterreference1}. Next, in view of the third item, no additional
common fixed point can ``bifurcate'' from $0 \in \C$ by perturbing $(h_1,h_2)$ due to the usual Argument principle.
Note also that the condition in the third item is always verified when $0 \in \C$ has multiplicity~$1$ as fixed point
of $W_i (\tilde{f}, \tilde{g}), \, W_j (\tilde{f}, \tilde{g})$. Since, apart from the $0 \in \C$,
$W_i (\tilde{f}, \tilde{g}), \, W_j (\tilde{f}, \tilde{g})$ share no fixed point, it follows from the second
part of the statement of Lemma~\ref{forlaterreference1} that
${\rm Fix}\, (W_i (\overline{h}_{1}^{-1} \circ f \circ \overline{h}_{1},
\overline{h}_{2}^{-1} \circ g \circ \overline{h}_{2})) \cap {\rm Fix}\, (W_j (\overline{h}_{1}^{-1} \circ f \circ \overline{h}_{1},
\overline{h}_{2}^{-1} \circ g \circ \overline{h}_{2})) = \{0\} \subset \C$ for every pair
$(\overline{h}_1, \overline{h}_2)$ sufficiently close to $(h_1,h_2)$.

Summarizing what precedes, we have proved the following.
\begin{lema}
\label{openessend1}
The set\, $\mathcal{U}_{i,j} \subset \diffalpha \times \diffalpha$ is open for every pair of incommensurable
words $W_i (a,b), \, W_j (a,b)$.\qed
\end{lema}

On the other hand, we shall also prove:

\begin{prop}
\label{disjointfixedpoints}
The set\, $\mathcal{U}_{i,j} \subset \diffalpha \times \diffalpha$ is dense for every pair of incommensurable
words $W_i (a,b), \, W_j (a,b)$.
\end{prop}

As soon as Proposition~\ref{disjointfixedpoints} is established, we are able to prove Theorem~A. Since the proof of
Proposition~\ref{disjointfixedpoints} is long and technical, we shall first derive Theorem~A deferring to the next
section the proof of Proposition~\ref{disjointfixedpoints}.

\begin{proof}[Proof of Theorem~A]
Fixed a pair $(W_i (a,b), W_j (a,b))$ in $\mathcal{I}$, i.e. a pair of incommensurable words, consider the above
defined set $\mathcal{U}_{i,j}$. According to Lemma~\ref{openessend1} and to Proposition~\ref{disjointfixedpoints},
the set $\mathcal{U}_{i,j}$ is open and dense in $\diffalpha \times \diffalpha$. Let us then define
$$
\mathcal{U} =  \bigcap_{(W_{i} (a,b), W_{j} (a,b)) \in \mathcal{I}} \mathcal{U}_{i, j}  \, .
$$
Clearly $\mathcal{U}$ is a $G_{\delta}$-dense subset of $\diffalpha \times \diffalpha$. By construction of
$\mathcal{U}$, it also clear that the stabilizer of every point $z \neq 0 \in D$ is either cyclic or trivial.
Therefore, the proof of Theorem~A is reduced to show the existence of a sequence of points
$\{ Q_n\}$, $Q_n \neq 0$ for every $n \in \N$, with the properties indicated in the statement of Theorem~A.

For this, note that $\mathcal{U}$ is contained in the set $\mathcal{V}$ provided by Theorem~\ref{teo_MRR}.
In fact, for every reduced word $W (a,b)$, the element $W (\tilde{f}, \tilde{g})$ has only isolated fixed
points in ${\rm Dom}_W (D)$ so long $(h_1, h_2) \in \mathcal{U}$. Therefore
$W (\tilde{f}, \tilde{g})$ cannot coincide with the identity on any connected component of its
domain of definition ${\rm Dom}_W (D)$. It follows, in particular, that
{\it the germ}\, of $\Gamma_{h_1,h_2}$ at $0 \in \C$ is a
non-solvable group. Since the germ of $\Gamma_{h_1,h_2}$ at $0 \in \C$ is not solvable, there are points in $D$ that are hyperbolic
fixed points for certain elements of $\Gamma_{h_1,h_2}$, cf. \cite{Frank1}.
Furthermore, the multipliers of these fixed points can {\it a priori}\, be fixed in a dense set of $\C$.
However, inasmuch there are infinitely many points whose stabilizers
contain a hyperbolic element, it may happen that all these points are contained in a single orbit of the
pseudogroup $\Gamma_{h_1,h_2}$.

To complete the proof of the theorem, we proceed as follows. Let $Q_1$ be a hyperbolic fixed point of some element
$W_1 (\tilde{f}, \, \tilde{g}) \in \Gamma_{h_1,h_2}$. In particular the stabilizer of $Q_1$ is not trivial and hence
it must be cyclic. Thus we can assume that $W_1 (\tilde{f}, \, \tilde{g})$
is the generator of the stabilizer of $Q_1$.
As mentioned above, the orbit of $Q_1$ by $\Gamma_{h_1,h_2}$ is constituted by points whose stabilizers contain
some hyperbolic element, namely a certain conjugate of $W_1 (\tilde{f}, \, \tilde{g})$. Consider for each
point $\gamma. Q_1$ in the $Q_1$-orbit the multipliers of elements in the stabilizer of $\gamma. Q_1$.
The collection of all multipliers obtained from points in the $Q_1$-orbit is then denoted by $M_1$.
Now note that $M_1$ is a discrete subset of $\C$, indeed, $M_1$ is nothing but a cyclic subgroup of
$\C^{\ast}$ generated by the derivative of $W_1 (\tilde{f}, \, \tilde{g})$ at $Q_1$. Thus,
after \cite{Morefrank}, there must exist another hyperbolic fixed point $Q_2$ whose multiplier lies away
from a neighborhood of $M_1$ in $\C$. In particular, the orbits of $Q_1$ and $Q_2$ must be disjoint. However,
the preceding argument applies again to ensure that the orbit of $Q_2$ yields another discrete set of multipliers
$M_2 \subset \C$. The construction can then be continued to yield infinitely many hyperbolic fixed points
with pairwise disjoint orbits. Theorem~A is proved.
\end{proof}

Corollary~B is an immediate consequence of what precedes.

\begin{proof}[Proof of Corollary B]
The construction detailed in Section~5 of \cite{MRR} allows us to translate information on
the topology of the leaves of the corresponding foliations into dynamical properties of the pseudogroup
$\Gamma_{h_1,h_2}$, and conversely. By means of this connection, the items~(4) and~(5) of Corollary~B turn out to be implied by
Theorem~A. The remaining items were already established in \cite{MRR}. An alternative possibility is to resort
to the general statements of \cite{marinmattei}. In any event the proof of Corollary~B is completed.
\end{proof}

\section{Proof of Proposition~\ref{disjointfixedpoints}}

The rest of the paper is entirely devoted to the proof of Proposition~\ref{disjointfixedpoints}. This proof will be
accomplished in this section whereas the proof of a more technical lemma on the existence of suitable perturbations
will be supplied only in the last section of this article.

Let us start by explaining the strategy for proving Proposition~\ref{disjointfixedpoints}. Fix two reduced
incommensurable words $W_i (a,b)$ and $W_j (a,b)$ and suppose we are
given a pair $(h_1,h_2) \in \diffalpha \times \diffalpha$. We need to find $(\overline{h}_1, \overline{h}_2)
\in \mathcal{U}_{i,j}$ arbitrarily close to $(h_1,h_2)$. The existence of the desired pair $(\overline{h}_1, \overline{h}_2)$
will be shown by successively approximating
$(h_1,h_2)$ by elements in $\diffalpha \times \diffalpha$ that will fulfil ``more and more'' the conditions needed
to belong to $\mathcal{U}_{i,j}$. This will be done so that,
after finitely many steps, a pair
$(\overline{h}_1, \overline{h}_2) \in \mathcal{U}_{i,j}$ will be found in a given $\varepsilon$-neighborhood
of $(h_1,h_2)$.

This goes as follows. First, by using the ``denseness part'' of the statement of Proposition~\ref{herewego1} applied
to both words $W_i (a,b), \, W_j (a,b)$, we see that arbitrarily close to $(h_1,h_2)$ there is
$(\overline{h}_1, \overline{h}_2)$ leading to a pseudogroup $\overline{\Gamma}_{\overline{h}_1, \overline{h}_2}$
whose elements $W_i (\overline{h}_{1}^{-1} \circ f \circ \overline{h}_{1},
\overline{h}_{2}^{-1} \circ g \circ \overline{h}_{2})$ and $W_j (\overline{h}_{1}^{-1} \circ f \circ \overline{h}_{1},
\overline{h}_{2}^{-1} \circ g \circ \overline{h}_{2})$ possess only finitely many fixed points in their
closed domains of definitions ${\rm Dom}_{W_i} (\overline{D}) , \, {\rm Dom}_{W_j} (\overline{D})$.
In other words, to abridge notations, we can assume without loss of generality that the initial local diffeomorphisms
are already such that the words $W_i (h_{1}^{-1} \circ f \circ h_{1}, \, h_{2}^{-1} \circ g \circ h_{2})
= W_i (\tilde{f}, \tilde{g})$ and $W_j (h_{1}^{-1} \circ f \circ h_{1}, \, h_{2}^{-1} \circ g \circ h_{2}) =
W_j (\tilde{f}, \tilde{g})$ have only finitely many fixed points as elements of $\overline{\Gamma}_{h_1,h_2}$.

Denote then by ${\rm Fix}\, (W_{i} (\tilde{f}, \tilde{g}))$ (resp. ${\rm Fix}\, (W_{j} (\tilde{f}, \tilde{g}))$)
the set of fixed points of $W_{i} (\tilde{f}, \tilde{g})$ (resp. $W_{j} (\tilde{f}, \tilde{g})$) in its (closed)
domain of definition as element of $\overline{\Gamma}_{h_1,h_2}$. These sets are both finite. The proof of
Proposition~\ref{disjointfixedpoints} is essentially reduced to checking that
$(h_1,h_2)$ can be approximated by a pair $(\overline{h}_1, \overline{h}_2)$
yielding corresponding local diffeomorphisms $W_i (\overline{h}_{1}^{-1} \circ f \circ \overline{h}_{1},
\overline{h}_{2}^{-1} \circ g \circ \overline{h}_{2}) = W_i (\overline{f}, \, \overline{g})$
and $W_j (\overline{h}_{1}^{-1} \circ f \circ \overline{h}_{1},
\overline{h}_{2}^{-1} \circ g \circ \overline{h}_{2}) = W_j (\overline{f}, \, \overline{g})$ having no common fixed
point other than $0 \in \C$. The condition that the multiplicity of $0 \in \C$ as fixed point of
$W_i (\overline{f}, \, \overline{g}), \, W_j (\overline{f}, \, \overline{g})$ should not change under perturbations is
a minor one and it will be dealt with below. Consider then the problem of showing that
$W_i (\tilde{f}, \tilde{g}), \, W_j (\tilde{f}, \tilde{g})$ can be perturbed so as not to have common fixed point other
than $0 \in \C$. Note that the construction of the desired perturbations is a problem that is naturally localized
at the mentioned common fixed points. To explain this assertion and clarify
the rest of our strategy to approach Proposition~\ref{disjointfixedpoints}, suppose
for example that the initial pair $(h_1,h_2)$ is such that $p \neq 0$ is the only common fixed point for
$W_{i} (\tilde{f}, \tilde{g}), \, W_{j} (\tilde{f}, \tilde{g})$ away from $0 \in \C$.
Consider then a small disc $B (\delta)$ about $p$. Also the (closed) domain of definition of
$W_{i} (\tilde{f}, \tilde{g})$ (resp. $W_{j} (\tilde{f}, \tilde{g})$) is going to be denoted by
${\rm Dom}_{W_i} (\overline{D})$ (resp. ${\rm Dom}_{W_j} (\overline{D})$). According to
Lemma~\ref{extensionoverclosure} (cf. Corollary~\ref{corobyproduct} and Proposition~\ref{herewego1}), we can choose a (closed)
neighborhood $\overline{U}_i^{\epsilon}$ of ${\rm Dom}_{W_i} (\overline{D})$
(resp. $\overline{U}_j^{\epsilon}$ of ${\rm Dom}_{W_j} (\overline{D})$) where the following holds:
\begin{itemize}
  \item $W_{i} (\tilde{f}, \tilde{g})$ (resp. $W_{j} (\tilde{f}, \tilde{g})$)
  has a holomorphic extension to some neighborhood of $\overline{U}_i^{\epsilon}$ (resp. $\overline{U}_j^{\epsilon}$).

  \item The fixed points of $W_{i} (\tilde{f}, \tilde{g})$ (resp. $W_{j} (\tilde{f}, \tilde{g})$) in
  $\overline{U}_i^{\epsilon}$ (resp. $\overline{U}_j^{\epsilon}$) are all contained in
  ${\rm Dom}_{W_i} (\overline{D})$ (resp. ${\rm Dom}_{W_j} (\overline{D})$).

\end{itemize}
Now consider the compact set $\overline{U}_j^{\epsilon} \cap [\overline{U}_i^{\epsilon} \setminus B (\delta)]$ and
note that for every point $z \in \overline{U}_j^{\epsilon} \cap [\overline{U}_i^{\epsilon} \setminus B (\delta)]$,
we have $\Vert W_{i} (\tilde{f}, \tilde{g}) (z) - W_{j} (\tilde{f}, \tilde{g})(z) \Vert \neq 0$. Since this set
is compact, it follows the existence of some $\tau >0$ such that we actually have
$\Vert W_{i} (\tilde{f}, \tilde{g}) (z) - W_{j} (\tilde{f}, \tilde{g})(z) \Vert \geq \tau >0$ for every
$z \in \overline{U}_j^{\epsilon} \cap [\overline{U}_i^{\epsilon} \setminus B (\delta)]$. Next the reader is reminded that
convergence in the analytic topology implies convergence of domains of definition as well as uniform convergence of maps on
the corresponding domains. Thus, recalling that ${\rm Dom}_{W_j} (\overline{D})$ is contained in the interior of
$\overline{U}_j^{\epsilon}$, by taking $(\overline{h}_1, \overline{h}_2)$ very close to $(h_1,h_2)$ the following holds:
\begin{itemize}
  \item[($\imath$)] The closed domain of definition of $W_i (\overline{f}, \, \overline{g})$ (resp.
  $W_j (\overline{f}, \, \overline{g})$) is contained in $\overline{U}_i^{\epsilon}$ (resp. $\overline{U}_j^{\epsilon}$).

  \item[($\imath \imath$)] The local diffeomorphism $W_i (\overline{f}, \, \overline{g})$ (resp.
  $W_j (\overline{f}, \, \overline{g})$ possesses a holomorphic extension to a neighborhood of
  $\overline{U}_i^{\epsilon}$ (resp. $\overline{U}_j^{\epsilon}$).
\end{itemize}
It follows from~($\imath$) and~($\imath \imath$) that $W_i (\overline{f}, \, \overline{g}) (z) - W_j (\overline{f}, \, \overline{g}) (z)$
is defined for every $z \in \overline{U}_j^{\epsilon} \cap [\overline{U}_i^{\epsilon} \setminus B (\delta)]$. Next, by
uniform convergence, it also follows that $\Vert W_i (\overline{f}, \, \overline{g}) (z) - W_j (\overline{f}, \, \overline{g}) (z)
\Vert \geq \tau/2 >0$ for every $z \in \overline{U}_j^{\epsilon} \cap [\overline{U}_i^{\epsilon} \setminus B (\delta)]$.
Therefore we conclude that the all possible {\it common fixed points}\, of $W_i (\overline{f}, \, \overline{g})$
and $W_j (\overline{f}, \, \overline{g})$ lie in $B(\delta)$.
Thus, the proof Proposition~\ref{disjointfixedpoints}
is essentially reduced to an analysis of $W_{i} (\tilde{f}, \tilde{g})$ and of $W_{j} (\tilde{f}, \tilde{g})$ on
a neighborhood of their common fixed points. Namely we need to show that these common fixed points can be split by
arbitrarily small perturbations of the initial local diffeomorphisms $(h_1,h_2)$. In the sequel, we shall provide
full detail for this construction.

Summarizing what precedes, we can assume that both $W_i (\overline{f}, \, \overline{g})$ and
$W_j (\overline{f}, \, \overline{g})$ have only finitely many fixed points in their closed domains of definition
denoted respectively by ${\rm Dom}_{W_i} (\overline{D})$ and ${\rm Dom}_{W_j} (\overline{D})$. Moreover, we can choose
a neighborhood $U_i^{\epsilon}$ of ${\rm Dom}_{W_i} (\overline{D})$ (resp. $U_j^{\epsilon}$ of ${\rm Dom}_{W_j} (\overline{D})$)
such that $W_i (\tilde{f}, \tilde{g})$ (resp. $W_j (\tilde{f}, \tilde{g})$) has a holomorphic extension to a neighborhood
of $\overline{U}_i^{\epsilon}$ (resp. $\overline{U}_j^{\epsilon}$). Furthermore the fixed points of
$W_i (\tilde{f}, \tilde{g})$ (resp. $W_j (\tilde{f}, \tilde{g})$) in
$\overline{U}_i^{\epsilon}$ (resp. $\overline{U}_j^{\epsilon}$) are all contained in
${\rm Dom}_{W_i} (\overline{D})$ (resp. ${\rm Dom}_{W_j} (\overline{D})$).

Recalling that our purpose is to find $(\overline{h}_1, \overline{h}_2)
\in \mathcal{U}_{i,j}$ arbitrarily close to $(h_1,h_2)$, we begin with the following lemma.

\begin{lema}
\label{herewego2}
To construct $(\overline{h}_1, \overline{h}_2)$, we can assume that the multiplicity associated to each fixed point of both
$W_{i} (\tilde{f}, \tilde{g})$ and $W_{j} (\tilde{f}, \tilde{g})$ does not change under perturbations.
\end{lema}

\begin{proof}
Consider the case of $W_{i} (\tilde{f}, \tilde{g})$.
Denote by $P_1, \ldots , P_s$ its fixed points. Suppose first that $W_{i} (\tilde{f}, \tilde{g})$
has no fixed point in the boundary of its domain of definition. Let $B_k(\delta)$ be a small ball about $P_k$, $k=1, \ldots ,s$
containing no other fixed point of $W_{i} (\tilde{f}, \tilde{g})$. Also denote by $N_k$ the
multiplicity of $P_k$ as fixed point of $W_{i} (\tilde{f}, \tilde{g})$. By continuity,
if $(\overline{h}_1,\overline{h}_2)$ is sufficiently close to $(h_1,h_2)$, then we have:
\begin{enumerate}
  \item The (closed) domain of definition of $W_{i} (\overline{h}_{1}^{-1} \circ f \circ
\overline{h}_{1}, \, \overline{h}_{2}^{-1} \circ g \circ \overline{h}_{2})$ is contained in $\overline{U}^{\epsilon}_i$.
  \item The diffeomorphism $W_{i} (\overline{h}_{1}^{-1} \circ f \circ
\overline{h}_{1}, \, \overline{h}_{2}^{-1} \circ g \circ \overline{h}_{2})$ has no fixed point in
$\overline{U}^{\epsilon}_i \setminus \bigcup_{k=1}^s B_k(\delta)$.
\end{enumerate}
Also, fixed $k$, the number of fixed points of $W_{i} (\overline{h}_{1}^{-1} \circ f \circ
\overline{h}_{1}, \, \overline{h}_{2}^{-1} \circ g \circ \overline{h}_{2})$ in $B_k(\delta)$ counted with their multiplicities is
precisely $N_k$, as it follows again from the Argument principle.
In particular, the multiplicity of a fixed point may decrease, but never increase, under perturbations. Thus, if there
are arbitrarily small perturbations for which each point $P_k$ splits into $N_k$
fixed points with multiplicity~$1$, then the statement
becomes immediate: we consider a first perturbation $(\overline{h}_1,\overline{h}_2)$ such that all fixed points become
of multiplicity~$1$ and then, we only need to approximate $(\overline{h}_1,\overline{h}_2)$ by elements
$(\overline{\overline{h}}_1, \overline{\overline{h}}_2)$ as above. Since, in this case,
every fixed point of
$W_{i} (\overline{h}_{1}^{-1} \circ f \circ \overline{h}_{1}, \, \overline{h}_{2}^{-1} \circ g \circ \overline{h}_{2})$
has multiplicity~$1$, this multiplicity will not change by perturbations as already seen.
The general case follows from this argument, we consider the least multiplicity that can be achieved for each fixed
point of $W_{i} (\tilde{f}, \tilde{g})$ by perturbing $(h_1,h_2)$. Being given by local minima,
these multiplicities cannot further decrease under perturbations. On the other hand, the previous general argument shows
that they cannot increase either.
Therefore they must remain constant what proves the lemma in this case.

Finally the argument when $W_{i} (\tilde{f}, \tilde{g})$ possesses fixed points in the boundary
of its domain of definition is essentially the same. These points are in finite number.
If by an arbitrarily small perturbation some of them fall
in the open domain and others fall away from the closed domain, then the situation is reduced to the
preceding case. Otherwise there are points
that remain in the boundary of the domain of definition of $W_{i} (\tilde{f}, \tilde{g})$ for
every sufficient small perturbation of $(h_1,h_2)$. The same argument above can then be applied to these fixed points.
\end{proof}

Owing to Lemma~\ref{herewego2}, we can assume without loss of generality that, in addition, the multiplicities
of the fixed points of $W_{i} (\tilde{f}, \tilde{g})$ and of $W_{j} (\tilde{f}, \tilde{g})$ do not change under perturbation
of $(h_1,h_2)$. Recalling that $W_{i} (\tilde{f}, \tilde{g})$ and $W_{j} (\tilde{f}, \tilde{g})$
have only isolated fixed points, the next lemma establishes that a common fixed point for
$W_{i} (\tilde{f}, \tilde{g})$ and $W_{j} (\tilde{f}, \tilde{g})$
can always be destroyed by arbitrarily small perturbations provided that $W_{i} (a,b),
\, W_{j} (a,b)$ are two incommensurable words. More precisely:

\begin{lema}
\label{herewego3}
Consider two incommensurable words $W_{i} (a,b), \, W_{j} (a,b)$ as above along with a given pair $(h_1,h_2) \in
\diffalpha \times \diffalpha$.
Suppose that $q \neq 0$ is a common fixed point for $W_{i} (\tilde{f}, \tilde{g})$ and $W_{j} (\tilde{f}, \tilde{g})$
(viewed as elements of $\Gamma_{h_1,h_2}$).
Let $B(\delta)$ be a small disc about $q$ containing no other fixed point of
$W_{i} (\tilde{f}, \tilde{g})$ and $W_{j} (\tilde{f}, \tilde{g})$. Then, arbitrarily close to
$(h_1,h_2)$, there is $(\overline{h}_1,\overline{h}_2)$ such that
$W_{i} (\overline{h}_1^{-1} \circ f \circ \overline{h}_{1}, \, \overline{h}_{2}^{-1}
\circ g \circ \overline{h}_{2})$ and $W_{j}
(\overline{h}_{1}^{-1} \circ f \circ \overline{h}_{1}, \, \overline{h}_{2}^{-1} \circ g \circ \overline{h}_{2})$
have no common fixed point in $B(\delta)$.
\end{lema}

After the preceding discussion,
Lemma~\ref{herewego3} is the main technical result needed for the proof of Proposition~\ref{disjointfixedpoints}.
We shall close this section with the proof of Proposition~\ref{disjointfixedpoints}. The next section will
be devoted to the constructions leading to the proof of Lemma~\ref{herewego3}

\begin{proof}[Proof of Proposition~\ref{disjointfixedpoints}]
The argument is now clear. Again denote by $P_1, \ldots , P_s$ the fixed points of, say,
$W_{i} (\tilde{f}, \tilde{g})$.
As explained in the proof of Lemma~\ref{herewego2}, modulo constructing a first perturbation of $(h_1,h_2)$,
we can assume that $W_{i} (\tilde{f}, \tilde{g})$
has no fixed points in the boundary of its domain of definition, unless the fixed points lying in this boundary remain
in it for every sufficiently small perturbation of
$(h_1,h_2)$. In the rest of the discussion, both possibilities will be treated together since there is no essential
difference between them.

Consider the above fixed neighborhood $U_i^{\epsilon}$ of the closed domain of definition of
$W_{i} (\tilde{f}, \tilde{g})$. Let $\delta >0$ very small be fixed. For every $k=1, \ldots ,s$, denote by $B_k (\delta)$
a small disc about the fixed point $P_k$. The choice of $\delta$ is dictated by the fact that the distance between
any pair of these discs must be strictly positive and by the fact that they should all be contained in a compact
part of $U_i^{\epsilon}$. By construction, $W_i (\tilde{f}, \tilde{g})$ has no
additional fixed point in $U_i^{\epsilon}$. Thus, if $(\overline{h}_1, \overline{h}_2)$ is close enough
to $(h_1,h_2)$, it follows that $W_i (\overline{f}, \, \overline{g}) = W_i (\overline{h}_1^{-1} \circ f
\circ \overline{h}_1, \overline{h}_2^{-1} \circ g \circ \overline{h}_2)$ satisfies the following conditions:
\begin{itemize}
  \item The closed domain of definition of $W_i (\overline{f}, \, \overline{g})$ is contained in $U_i^{\epsilon}$.
  \item $W_i (\overline{f}, \, \overline{g})$ has no fixed point in
  $\overline{U}_i^{\epsilon} \setminus \bigcup_{k=1}^s B_k (\delta)$.
\end{itemize}
Analogous conclusions hold for $W_j (\overline{f}, \, \overline{g}) = W_j (\overline{h}_1^{-1} \circ f
\circ \overline{h}_1, \overline{h}_2^{-1} \circ g \circ \overline{h}_2)$. In the sequel all perturbations will be chosen
small enough to guarantee that the final ``perturbed'' diffeomorphisms still satisfy the above conditions for
$W_i (\overline{f}, \, \overline{g})$ and for $W_j (\overline{f}, \, \overline{g})$.

In view of what precedes, and given that the multiplicities
of $P_1, \ldots , P_s$ do not change under perturbations, the number $s$ of fixed points of
$W_{i} (h_{1}^{-1} \circ f \circ h_{1}, \, h_{2}^{-1} \circ g \circ h_{2})$ will not change
under sufficiently small perturbations of $(h_1,h_2)$ (owing again to the Argument principle).

Consider again the above defined discs $B_k (\delta)$, $k=1, \ldots ,s$. Recall that every sufficiently
small perturbation $(\tilde{h}_1, \tilde{h}_2)$ of $(h_1,h_2)$ leads to a new local diffeomorphism
$W_i (\tilde{h}_1^{-1} \circ f \circ \tilde{h}_1, \,
\tilde{h}_2^{-1} \circ g \circ \tilde{h}_2)$ having exactly one fixed point in each small disc $B_k (\delta)$ and no
fixed point in the complement of $\bigcup_{k=1}^s B_k (\delta)$ in the closed domain of definition
of $W_i (\tilde{h}_1^{-1} \circ f \circ \tilde{h}_1, \, \tilde{h}_2^{-1} \circ g \circ \tilde{h}_2)$ itself.
Starting from $k=1$, suppose that $P_1$ is a common fixed point for $W_i (\overline{f}, \, \overline{g})$ and
$W_j (\overline{f}, \, \overline{g})$. Then Lemma~\ref{herewego3}
allows us to find an arbitrarily small perturbation $(\overline{h}_{1,1}, \overline{h}_{2,1})$ of $(h_1,h_2)$ such that
$W_{i} (\overline{h}_{1,1}^{-1} \circ f \circ \overline{h}_{1,1}, \, \overline{h}_{2,1}^{-1} \circ g \circ \overline{h}_{2,1})$
and $W_{j} (\overline{h}_{1,1}^{-1} \circ f \circ \overline{h}_{1,1}, \, \overline{h}_{2,1}^{-1} \circ g \circ \overline{h}_{2,1})$
have no longer a common fixed point in $B_1 (\delta)$.
Furthermore, if $(\overline{h}_{1,1}, \overline{h}_{2,1})$ is close enough to $(h_1,h_2)$, then
$W_{i} (\overline{h}_{1,1}^{-1} \circ f \circ \overline{h}_{1,1}, \, \overline{h}_{2,1}^{-1} \circ g \circ \overline{h}_{2,1})$
still satisfies the previous conditions regarding its own fixed points, which will now be denoted by
$P_{1,1}, \ldots , P_{s,1}$. In particular, each $P_{k,1}$ lies in $B_k (\delta)$. By construction $P_{1,1}$ is not
a common fixed point for
$W_{i} (\overline{h}_{1,1}^{-1} \circ f \circ \overline{h}_{1,1}, \, \overline{h}_{2,1}^{-1} \circ g \circ \overline{h}_{2,1})$
and $W_{j} (\overline{h}_{1,1}^{-1} \circ f \circ \overline{h}_{1,1}, \, \overline{h}_{2,1}^{-1} \circ g \circ \overline{h}_{2,1})$.
Thus, if no point $P_{k,1}$, $k=2, \ldots ,s$ turns out to be
fixed by $W_{j} (\overline{h}_{1,1}^{-1} \circ f \circ \overline{h}_{1,1}, \, \overline{h}_{2,1}^{-1} \circ g \circ \overline{h}_{2,1})$
then the statement is proved.

Thus let us suppose that $P_{2,1}$ is fixed also by
$W_{j} (\overline{h}_{1,1}^{-1} \circ f \circ \overline{h}_{1,1}, \, \overline{h}_{2,1}^{-1} \circ g \circ \overline{h}_{2,1})$.
By using again Lemma~\ref{herewego3} we can find a new perturbation $(\overline{h}_{1,2}, \overline{h}_{2,2})$ of
$(\overline{h}_{1,1}, \overline{h}_{2,1})$ so that
$W_{i} (\overline{h}_{1,2}^{-1} \circ f \circ \overline{h}_{1,2}, \, \overline{h}_{2,2}^{-1} \circ g \circ \overline{h}_{2,2})$ and
$W_{j} (\overline{h}_{1,2}^{-1} \circ f \circ \overline{h}_{1,2}, \, \overline{h}_{2,2}^{-1} \circ g \circ \overline{h}_{2,2})$
have no longer a common fixed point in $B_2 (\delta)$. Furthermore,
modulo choosing this perturbation sufficiently small, the following conditions can again be ensured:
\begin{itemize}
\item No common fixed point for $W_{i} (\overline{h}_{1,2}^{-1} \circ f \circ \overline{h}_{1,2}, \, \overline{h}_{2,2}^{-1} \circ g \circ \overline{h}_{2,2})$ and
$W_{j} (\overline{h}_{1,2}^{-1} \circ f \circ \overline{h}_{1,2}, \, \overline{h}_{2,2}^{-1} \circ g \circ \overline{h}_{2,2})$
is produced in $B_1 (\delta)$.

\item $W_{i} (\overline{h}_{1,1}^{-1} \circ f \circ \overline{h}_{1,1}, \, \overline{h}_{2,1}^{-1} \circ g \circ \overline{h}_{2,1})$
still has exactly $s$ fixed points, denoted by $P_{1,2}, \ldots , P_{s,2}$. Besides, for every $k=1, \ldots ,s$,
the fixed point $P_{k,2}$ belongs to the disc $B_k (\delta)$.
\end{itemize}
In particular, after this second perturbation,
$W_{i} (\overline{h}_{1,2}^{-1} \circ f \circ \overline{h}_{1,2}, \, \overline{h}_{2,2}^{-1} \circ g \circ \overline{h}_{2,2})$ and
$W_{j} (\overline{h}_{1,2}^{-1} \circ f \circ \overline{h}_{1,2}, \, \overline{h}_{2,2}^{-1} \circ g \circ \overline{h}_{2,2})$
can have at most $s-2$ common fixed points. By inductively continuing this argument, we shall eventually obtain a perturbation
$(\overline{h}_1, \overline{h}_2) = (\overline{h}_{1,s}, \overline{h}_{2,s})$ of $(h_1,h_2)$ satisfying the
condition required in the statement. The proof of Proposition~\ref{disjointfixedpoints} is over.
\end{proof}

\section{Constructing analytic perturbations: proof of Lemma~\ref{herewego3}}

In this last section we shall introduce some perturbation techniques leading to the proof of Lemma~\ref{herewego3}.
First, we assume that the conditions used in the previous section still hold in the present context. This means
that $W_i (\tilde{f}, \tilde{g})$ and $W_j (\tilde{f}, \tilde{g})$ have only isolated fixed points and, furthermore,
that the multiplicity of these fixed points do not change under perturbations of $(h_1,h_2)$.

The approach to the proof of Lemma~\ref{herewego3} begins
with some simple reductions in the statement. First it will be proved that,  if $q \in \C$, $q \ne 0$, is a fixed
point for $W_{i} (\tilde{f}, \tilde{g})$, then we can perturb $(h_1,h_2)$ into
$(\overline{h}_1, \overline{h}_2)$ so that $q$ is no longer fixed by
$W_{j} (\overline{h}_{1}^{-1} \circ f \circ \overline{h}_{1}, \, \overline{h}_{2}^{-1} \circ g \circ \overline{h}_{2})$.
Assuming that $q$ is a common fixed point for both $W_{i} (\tilde{f}, \tilde{g})$ and
$W_j (\tilde{f}, \tilde{g})$, the proof of Lemma~\ref{herewego3}
amounts to checking that such perturbation can be applied to, say, $W_{j} (\tilde{f}, \tilde{g})$ at $q$
while keeping the point $q$ fixed by
$W_{i} (\overline{h}_{1}^{-1} \circ f \circ \overline{h}_{1}, \, \overline{h}_{2}^{-1} \circ g \circ \overline{h}_{2})$.
As already explained, since the multiplicity of $q$ as fixed point of $W_{i} (\tilde{f}, \tilde{g})$ does not change
under perturbations of $(h_1,h_2)$, it follows that for every pair $(\overline{h}_1, \overline{h}_2)$ sufficiently close to
$(h_1,h_2)$, $W_{i} (\overline{h}_{1}^{-1} \circ f \circ \overline{h}_{1}, \, \overline{h}_{2}^{-1} \circ g \circ \overline{h}_{2})$
will still have a unique fixed point on a fixed neighborhood of $q$. Thus, if $q$ remains fixed by the perturbed
diffeomorphism $W_{i} (\overline{h}_{1}^{-1} \circ f \circ \overline{h}_{1}, \, \overline{h}_{2}^{-1} \circ g \circ \overline{h}_{2})$,
it follows that this diffeomorphism cannot have additional fixed points in the neighborhood in question. In other words,
$W_{i} (\overline{h}_{1}^{-1} \circ f \circ \overline{h}_{1}, \, \overline{h}_{2}^{-1} \circ g \circ \overline{h}_{2})$ and
$W_{j} (\overline{h}_{1}^{-1} \circ f \circ \overline{h}_{1}, \, \overline{h}_{2}^{-1} \circ g \circ \overline{h}_{2})$
have no common fixed point in $B(\delta)$ as desired.

Let us begin by making accurate the first statement above.

\begin{lema}
\label{directuse}
Consider the element $W (\tilde{f},\, \tilde{g})$ viewed as belonging to the pseudogroup $\Gamma_{h_1,h_2}$.
Suppose that $q \ne 0$ lies in the (open) domain of definition of $W (\tilde{f},\, \tilde{g})$. Suppose also that
$W (\tilde{f},\, \tilde{g}) (q) = q$.
Then, there is $(h_{1,\ast}, h_{2,\ast}) \in \diffalpha \times \diffalpha$ arbitrarily close to $(h_1,h_2)$
and such that the following holds:
\begin{itemize}
\item[(a)] $q$ lies in the (open) domain of definition of $W (h_{1,\ast}^{-1} \circ f \circ h_{1,\ast}, \,
h_{2,\ast}^{-1} \circ g \circ h_{2,\ast})$ viewed as element of the pseudogroup generated by
$h_{1,\ast}^{-1} \circ f \circ h_{1,\ast}, \, h_{2,\ast}^{-1} \circ g \circ h_{2,\ast}$ on $D$.

\item[(b)] $W  (h_{1,\ast}^{-1} \circ f \circ h_{1,\ast}, \, h_{2,\ast}^{-1} \circ g \circ h_{2,\ast}) (q) \neq q$.
\end{itemize}
\end{lema}

\begin{proof}
Condition~(a) is always satisfied provided that $(h_{1,\ast}, h_{2,\ast})$ is very close to $(h_1,h_2)$. Thus we only need
to prove that, arbitrarily close to $(h_1,h_2)$, there is
$(h_{1,\ast}, h_{2,\ast}) \in \diffalpha \times \diffalpha$ satisfying condition~(b).

Consider the spelling of $W (a,b)$ under the form $W (a,b) = \vartheta_l^{r_l} \ast \cdots \ast \vartheta_1^{r_1}$.
The proof of the existence
of $(h_{1,\ast}, h_{2,\ast})$ satisfying condition~(b) and arbitrarily close to $(h_1,h_2)$ is going to be carried out by
induction on $l$. Suppose first that $l$ equals to~$1$. In this case, the statement follows at once from the fact that $f$ does not
have a Cremer point at $0 \in \C$, cf. Lemma~\ref{nocremer}.

By inducting on the length of the words, the proposition can be assumed to hold for words of length $1, \ldots ,l-1$.
We need to show that it also holds for words of length~$l$. First consider the itinerary $q=q_0, \ldots , q_{l-1}, q_l$
of $q$ under $W (h_1^{-1} \circ f \circ h_1, \, h_2^{-1} \circ g \circ h_2) = W (\tilde{f}, \tilde{g})$.
By assumption we have $q=q_0 =q_l$. The induction assumption allows us to suppose that the points $q_0 , \ldots , q_{l-1}$ are
pairwise distinct. Indeed, given $0 \leq k_1 < k_2 < l$, we have that
$q_{k_2} = W' (\tilde{f}, \tilde{g}) (q_{k_1})$ where $W'(a,b)$ is a word whose length
is at most $l-1$. Thus, by the induction assumption, $(h_1,h_2)$ can be perturbed into
$(h_{1, \ast} ,h_{2, \ast}) \in \diffalpha \times \diffalpha$ so as to satisfy
$q_{k_2}  = W' (h_{1, \ast}^{-1} \circ f \circ h_{1, \ast} , \, h_{2, \ast}^{-1} \circ g \circ h_{2, \ast} ) (q_{k_1})
\neq q_{k_1}$. Since, once obtained, the condition
$q_{k_2}  = W' (h_{1, \ast}^{-1} \circ f \circ h_{1, \ast} , \, h_{2, \ast}^{-1} \circ g \circ h_{2, \ast} ) (q_{k_1}) \neq q_{k_1}$
is open, the fact that there are only finitely many words $W' (a,b)$ that need to be considered
allows us to construct a first perturbation
$(\tilde{h}_{1,\ast}, \tilde{h}_{2,\ast}) \in \diffalpha \times \diffalpha$
of $(h_1,h_2)$ so that the itinerary of $q=q_0$ by
$W (\tilde{h}_{1,\ast}^{-1} \circ f \circ \tilde{h}_{1, \ast}, \, \tilde{h}_{2,\ast}^{-1} \circ g
\circ \tilde{h}_{2,\ast})$ satisfies the required condition. In other words, we can assume without loss
of generality that the itinerary $q=q_0, \ldots , q_{l-1}, q_l$ of $q$ under
$W (h_1^{-1} \circ f \circ h_1, \, h_2^{-1} \circ g \circ h_2)$ is such that the points $q_0 , \ldots , q_{l-1}$ are pairwise distinct.

Let us now construct pairs of local diffeomorphisms $(h_{1,\ast}, h_{2,\ast}) \in
\diffalpha \times \diffalpha$ arbitrarily close to $(h_1,h_2)$
and such that $W (h_{1,\ast}^{-1} \circ f \circ h_{1,\ast}, \, h_{2,\ast}^{-1} \circ g \circ h_{2,\ast}) (q) \neq q$.
First, since $W (a,b) = \vartheta_l^{r_l} \ast \cdots \ast \vartheta_1^{r_1}$, with $l \geq 2$, we shall assume that
$\vartheta_1$ takes on the value $a$ (with $r_1 > 0$) and that $\vartheta_l$ takes on the value $b$ (with $r_l> 0$).
Note that, if the word $W (a,b)$ is such that $\vartheta_1, \vartheta_l$ takes on the same value ($a$ or $b$),
then $W (f,g)$ is conjugate to a word of smaller length and the desired conclusion can immediately be derived.

Let $P$ be a polynomial such that $P(q_0) = \cdots = P (q_{l-2}) =0$ and $P (q_{l-1}) \neq 0$. Since $\vartheta_l$
takes on the value $b$, we set
$$
h_{1,t} = h_1 \, \; {\rm and} \, \ h_{2,t} = h_2 + tz^{\alpha+1}P
$$
where $t \in [0,1]$. Clearly $h_{2,t}$ converges to $h_2$ in the analytic topology when $t \rightarrow 0$
and $h_{2,t} \in \diffalpha$ for every $t \in [0,1]$. Therefore, to conclude the proof, it suffices to show that
$W (h_{1,t}^{-1} \circ f \circ h_{1,t}, \, h_{2,t}^{-1} \circ g \circ h_{2,t}) (q) \neq q$ for arbitrarily small
$t > 0$ (strictly). As already observed, for $t$ sufficiently small $q$ belongs to the domain of definition of
$W(h_{1,t}^{-1} \circ f \circ h_{1,t}, \, h_{2,t}^{-1} \circ g \circ h_{2,t})$ viewed
as an element of the pseudogroup generated on the open disc $D$ by
$h_{1,t}^{-1} \circ f \circ h_{1,t}, \, h_{2,t}^{-1} \circ g \circ h_{2,t}$.
The corresponding itinerary is going to be denoted by $q = q_{0,t}, \ldots , q_{l-2,t}, q_{l-1,t}$ and
$q_{l,t}^{l} = h_{2,t}^{-1} \circ g^{r_l} \circ h_{2,t} ( q_{l-1,t})$. By construction, it follows that
$q_k =q_{k,t}$ for $k=0,1, \ldots ,l-1$. However, $h_{2,t} (q_{l-1}) = h_{2,t} (q_{l-1,t}) \neq h_2 (q_{l-1,t})$.
Now the assumption concerning the injective character of both $h_2^{\pm 1}, g^{\pm 1}$ on the domains in question
implies that $q =q_0 = q_l = W(h_1^{-1} \circ f \circ h_1, \, h_2^{-1} \circ g
\circ h_2) (q_0) \neq W(h_{1,t}^{-1} \circ f \circ h_{1,t}, \, h_{2,t}^{-1} \circ g \circ h_{2,t}) (q_0)$
for every $t > 0$ sufficiently small. The lemma is proved.
\end{proof}

As already been mentioned, our strategy consists of showing that perturbations as in Lemma~\ref{directuse} can be applied
to $W_j (\tilde{f}, \tilde{g})$ while keeping $q$ as a fixed point of
$W_{i} (h_{1,\ast}^{-1} \circ f \circ h_{1,\ast}, \, h_{2,\ast}^{-1} \circ g \circ h_{2,\ast})$. The rest of the material
goes in this direction.

Let us write $W_{i} (a,b) = \vartheta_{l}^{r_l} \ast \cdots \ast \vartheta_1^{r_1}$ and $W_{j} (a,b) = \vartheta_{m}^{s_m}
\ast \cdots \ast \vartheta_1^{s_1}$. Note that $\vartheta_1$ on $W_{i}$ does not necessarily take on the same value of
$\vartheta_1$ on $W_{j}$. Nonetheless, although the same notation is used for simplicity, throughout the text, each time
we refer to $\vartheta_1$ it will be explicitly mentioned if we are considering $\vartheta_1$ on $W_{i}$ or in $W_{j}$.
Modulo re-labeling these two words, we always assume that $m \leq l$. The construction of the required perturbation
$(\overline{h}_1, \overline{h}_2)$ will be carried out by induction on $m$, i.e. on the length of the shorter word. To initialize
the induction, note that for every pair $(h_1,h_2) \in \diffalpha \times \diffalpha$,
the map $W_{j} (h_{1}^{-1} \circ f \circ h_{1}, \, h_{2}^{-1} \circ g \circ h_{2})$ has no isolated fixed points provided
that $m=1$ as it follows from
the fact that none of the local diffeomorphisms $f, g$ has a Cremer point at the origin, cf. Lemma~\ref{nocremer}.
The statement is then immediately true regardless of the value of $l$.
Therefore, by means of the induction, we assume that the existence of the desired perturbation $(\overline{h}_1, \overline{h}_2)$ was
already established for $1, \ldots ,m-1$ and every $l \in \N$. All we need to prove is the existence of
$(\overline{h}_1, \overline{h}_2)$ arbitrarily close to $(h_1,h_2)$ such that $q$ remains fixed by exactly one of the
words $W_{i} (\overline{h}_{1}^{-1} \circ f \circ \overline{h}_1, \, \overline{h}_{2}^{-1} \circ g \circ \overline{h}_{2})$
and $W_{j} (\overline{h}_{1}^{-1} \circ f \circ \overline{h}_1, \, \overline{h}_{2}^{-1} \circ g \circ \overline{h}_{2})$.

\begin{lema}
\label{minorlemma1}
We can assume that $W_i (a,b)$ (resp. $W_j (a,b)$) is not a power of some (third) word $W_3 (a,b)$
verifying $W_3 (h_{1}^{-1} \circ f \circ h_{1}, \, h_{2}^{-1} \circ g \circ h_{2}) (q) = q$.
\end{lema}

\begin{proof}
Consider first $W_i (a,b)$ and suppose it is a power of $W_3 (a,b)$ verifying $W_3 (h_{1}^{-1} \circ f \circ h_{1},
\, h_{2}^{-1} \circ g \circ h_{2}) (q) =q$. In this situation, to establish Lemma~\ref{herewego3},
it suffices to work with $W_3 (a,b)$ and $W_j (a,b)$. In fact, if
the statement is verified for these two latter words, the ``new'' fixed point $\overline{q}$ of
$W_3 (\overline{h}_{1}^{-1} \circ f \circ \overline{h}_1, \, \overline{h}_{2}^{-1} \circ g \circ \overline{h}_{2})$
is automatically a fixed point for
$W_i (\overline{h}_{1}^{-1} \circ f \circ \overline{h}_1, \, \overline{h}_{2}^{-1} \circ g \circ \overline{h}_{2})$
as well and, hence, it is the unique fixed point of
$W_i (\overline{h}_{1}^{-1} \circ f \circ \overline{h}_1, \, \overline{h}_{2}^{-1} \circ g \circ \overline{h}_{2})$
in $B(\delta)$.

The case of $W_j (a,b)$ is automatic: if $W_j (a,b)$ is a power of $W_3 (a,b)$ verifying $W_3 (h_{1}^{-1} \circ f
\circ h_{1}, \, h_{2}^{-1} \circ g \circ h_{2}) (q) =q$, then the induction assumption ensures the existence
of an arbitrarily small perturbation $(h_{1,\ast}, h_{2,\ast})$ such that
$W_i (\overline{h}_{1}^{-1} \circ f \circ \overline{h}_1, \, \overline{h}_{2}^{-1} \circ g \circ \overline{h}_{2}) (q) \neq q$
while $W_3 (\overline{h}_{1}^{-1} \circ f \circ \overline{h}_1, \, \overline{h}_{2}^{-1} \circ g \circ \overline{h}_{2}) (q) =q$.
Clearly $q$ must still be the unique fixed point of
$W_j (\overline{h}_{1}^{-1} \circ f \circ \overline{h}_1, \, \overline{h}_{2}^{-1} \circ g \circ \overline{h}_{2})$
in $B (\delta)$ so that the lemma follows.
\end{proof}

Next consider the itinerary $q = q_0^1, \ldots, q_l^1 = q_0^1$ of $q$ under $W_i (\tilde{f}, \tilde{g})$
and the itinerary $q = q_0^2, \ldots, q_m^2 = q_0^2$ of $q$ under $W_j (\tilde{f}, \tilde{g})$.

\begin{lema}
\label{minorlemma2}
We can suppose, without loss of generality, that $q_k^1 \neq q_0^1$ for $k=1, \ldots , l-1$ and that $q_k^2 \neq q_0^2$ for
$k=1, \ldots , m-1$.
\end{lema}

\begin{proof}
It follows from the same argument of the proof of Lemma~\ref{minorlemma1}. In fact, if this condition is not verified, the word
$W_i (a,b)$ (resp. $W_j (a,b)$) can be split into shorter words and it will be enough to work with these latter words to
eventually arrive to a situation where the condition in question is satisfied.
\end{proof}

Nonetheless, with the preceding notations, Lemma~\ref{directuse} {\it does not allow us}\, to suppose that the points
$q = q_0^1, \ldots , q_{l-1}^1$ are pairwise disjoint since, when carrying out the perturbations described in this lemma,
it may happen that the condition $q_0^1 = q_l^1$ becomes no longer fulfilled. Analogous considerations apply to the
itinerary $q = q_0^2, \ldots, q_m^2 = q_0^2$. For this reason, we shall need a version of Lemma~\ref{directuse}
adapted to the present setting. However, to state this result, some new terminology is needed.

As always all words $W (a,b)$ are supposed to be non-empty and reduced. A word $W (a,b)$ of length~$l$ is said to be {\it a
conjugate of type~$1$}\, of a shorter word if there are words $W_1 (a,b)$ and $W_2 (a,b)$ such that the spelling of $W(a,b)$
has the form $[W_1 (a,b)]^{-1} \ast W_2 (a,b) \ast W_1 (a,b)$, where the concatenation ``$\ast$'' leads to no simplification.
This last assumption implies that the length of $W_2 (a,b)$ plus twice the length of $W_1 (a,b)$ equals the length of
$W (a,b)$, i.e. it equals~$l$. Therefore, a word $W(a,b)$ of length~$l$ is a conjugate of type~1 of a shorter word if
the usual spelling $W (a,b) = \vartheta_{l}^{r_l} \ast \cdots \ast \vartheta_1^{r_1}$ of $W(a,b)$ has the form
$$
[\vartheta_{l}^{r_l} \ast \cdots \ast \vartheta_{l-s_0 +1}^{r_{l-s_0 +1}}]
\ast [\vartheta_{l-s_0}^{r_{l-s_0}} \ast \cdots \ast \vartheta_{s_0 +1}^{r_{s_0 +1}}] \ast [\vartheta_{s_0}^{r_{s_0}} \ast \cdots \ast \vartheta_1^{r_1}]
$$
where $[\vartheta_{l}^{r_l} \ast \cdots \ast \vartheta_{l-s_0 +1}^{r_{l-s_0 +1}}]$ represents the inverse of
$[\vartheta_{s_0}^{r_{s_0}} \ast \cdots \ast \vartheta_1^{r_1}]$. Naturally this construction does not affect the
domain of definition of $W (a,b)$ as element
of the corresponding pseudogroup. The word $W_2 (a,b)$ is then a (shorter) conjugate of $W (a,b)$. Naturally this definition does
not unequivocally characterize $W_1 (a,b), \,  W_2 (a,b)$. However it does characterize the {\it minimal conjugate of type~$1$}\,
$W_2 (a,b)$ of $W(a,b)$ which corresponds to $W_1 (a,b), \,  W_2 (a,b)$ as above such that $W_2 (a,b)$ is not a conjugate of
type~$1$ of a shorter word.

Let now $W (a,b)= [W_1 (a,b)]^{-1} \ast W_2 (a,b) \ast W_1 (a,b)$ where $W_2$ is the minimal conjugate of type~$1$ of $W (a,b)$.
Setting $W_2 (a,b) = \vartheta_{l-s_0}^{r_{l-s_0}} \ast \cdots \ast \vartheta_{s_0 +1}^{r_{s_0 +1}}$, the {\it minimal conjugate}\,
$W_4 (a,b)$ of $W (a,b)$ is defined as follows:
\begin{enumerate}
\item If $\vartheta_{l-s_0}$ and $\vartheta_{s_0+1}$ take on different values (negative exponents allowed),
then the minimal conjugate $W_4 (a,b)$ of $W (a,b)$ (or equivalently of $W_2 (a,b)$) is $W_2 (a,b)$ itself.

\item If both $\vartheta_{l-s_0}$ and $\vartheta_{s_0+1}$ take on the same value, say~$a$,
then the minimal conjugate $W_4 (a,b)$ of $W (a,b)$ (or equivalently of $W_2 (a,b)$) is given by
$$
\vartheta_{l-s_0}^{r_{l-s_0} - r_{s_0 +1}} \ast \cdots \ast \vartheta_{s_0 +2}^{r_{s_0 +2}}
$$
\end{enumerate}

In particular, a word $W (a,b) = \vartheta_{l}^{r_l} \ast \cdots \ast \vartheta_{1}^{r_{1}}$ coinciding with its own
minimal conjugate must be
such that $\vartheta_1$ and $\vartheta_l$ take on different values. The preceding definition is coherent in the sense that
$\vartheta_{s_0+2}$ and $\vartheta_{l-s_0}$ takes on different values. Indeed, $r_{l-s_0} - r_{s_0 +1} \neq 0$ since
$W_2 (a,b)$ is the minimal conjugate of type~$1$ of $W (a,b)$.

\begin{obs}
\label{minimalconjugateandsoon}
{\rm Although in the definition of {\it conjugate of type~$1$}\, the concatenation leads to no simplification,
the same does not necessarily occurs with the definition of {\it minimal conjugate}. For example, consider the
reduced word $W(a,b) = a^{-1} ba^2$. This word is not a conjugate of type~$1$ of a shorter word. In fact, albeit
$W(a,b)$ admits the spelling $W(a,b) = a^{-1} \ast (ba) \ast a$, this spelling leads to a simplification. Nonetheless,
the minimal conjugate $W_4 (a,b)$ of $W (a,b)$ does not coincide with $W(a,b)$ itself. First, note that $\vartheta_1$
and $\vartheta_3$ take on the same value, namely~$a$. Since $s_0$ (in the definition of the minimal conjugate of
type~$1$) is equal to {\it zero}, the length of the minimal conjugate should be equal to $l-s_0 -(s_0 +2) +1$, i.e.
equal to two. Indeed, the minimal conjugate of $W(a,b)$ coincides with the word $ab$ since $W(a,b)$ can be written
under the form $W (a,b) = a^{-2} \ast (ab) \ast a^2$, where a simplification on the concatenation~``$\ast$'' can
be used. This type of simplification is unique and it occurs only once.}
\end{obs}

With the previous notations, a version of Lemma~\ref{directuse} adapted to the present setting is as follows.

\begin{lema}\label{herewego4}
Let $W (a,b)$ be a word of length~$l$ and consider a point $q$ in the domain of definition of $
W (\tilde{f}, \tilde{g})=W (h_{1}^{-1} \circ f \circ h_{1}, \, h_{2}^{-1} \circ g \circ h_{2})$ viewed as
element of $\Gamma_{h_1,h_2}$ . Suppose that $W (\tilde{f}, \tilde{g})$ has only isolated fixed points and that
the multiplicity of each of these fixed points does not change under perturbations of $(h_1,h_2)$.
Denote by $q = q_0 , \ldots , q_l$ the itinerary of $q$ under $l$ and assume that $q_0 \not\in \{q_1, \ldots , q_{l-1}\}$.
Then arbitrarily close to $(h_1,h_2)$, there is $(\tilde{h}_1, \tilde{h}_2)$ such that the itinerary $q=\tilde{q}_0, \tilde{q}_1,
\ldots , \tilde{q}_l$ of $q$ under $W (\tilde{h}_{1}^{-1} \circ f \circ \tilde{h}_{1}, \, \tilde{h}_{2}^{-1}
\circ g \circ \tilde{h}_{2})$ satisfies the following conditions:
\begin{enumerate}

\item Suppose $q_0 \neq q_l$. Then $\tilde{q_l} =q_l$ and the points $q=q_0, \tilde{q}_1, \ldots , \tilde{q}_{l-1}$ are pairwise distinct.

\item Suppose $q_0 = q_l$ and $W (a,b)$ is not a conjugate of a shorter word. Then again $\tilde{q_l} =q_l =q_0$ and
the points $q=q_0, \tilde{q}_1, \ldots , \tilde{q}_{l-1}$ are pairwise distinct.

\item Suppose $q_0 =q_l$ and $W(a,b) = [W_3 (a,b)]^{-1} \ast W_4 (a,b) \ast W_3 (a,b)$, where $W_4(a,b)$
represents the minimal conjugate of $W (a,b)$. Denote the length of $W_3( a,b)$ (resp. $W_4 (a,b)$) by $s_0$ (resp. $l'$).
If the natural assumption that $q_{s_0}$ is different from all the points $q_{s_0+1}, \ldots , q_{s_0+l'-1}$ is added,
then the points $q = \tilde{q}_0, \tilde{q}_1 \ldots , \tilde{q}_{s_0+ l'-1}$ are pairwise distinct.
\end{enumerate}
\end{lema}

First let us make some comments concerning item~(3) of Lemma~\ref{herewego4}. Consider words $W (a,b), \, W_3 (a,b)$ and
$W_4 (a,b)$ as in item~(3) of the previous lemma. Denote by $l$ (resp. $s_0, \, l'$) the length of the word $W (a,b)$
(resp. $W_3 (a,b), \, W_4 (a,b)$). Two cases may occur.
\begin{itemize}
\item[(a)] The minimal conjugate $W_4 (a,b)$ coincides with the minimal conjugate of type~$1$ of $W (a,b)$.
In this case $l'+2s_0=l$, i.e. the concatenation~``$\ast$'' leads to no simplification in the spelling of
$[W_3 (a,b)]^{-1} \ast W_4 (a,b) \ast W_3 (a,b)$.

\item[(b)] The minimal conjugate $W_4 (a,b)$ does not coincide with the minimal conjugate of type~$1$ of $W (a,b)$.
In this case $l'+2s_0=l+1$, i.e. the concatenation~``$\ast$'' leads to a (unique) simplification in the spelling of
$[W_3 (a,b)]^{-1} \ast W_4 (a,b) \ast W_3 (a,b)$.
\end{itemize}
Next, consider the itinerary $q = q_0, \, \ldots, \, q_l$ of $q$ under $W (\tilde{f}, \tilde{g})$
and the itinerary $q = q_0', \, \ldots , \, q_p'$ of $q$ under $([W_3]^{-1} \ast W_4 \ast W_3)
(\tilde{f}, \tilde{g})$, where $p = l$ or $l+1$ according to we are in case~(a) or in case~(b).
In the first case, the itinerary $q_0, \, \ldots, \, q_l$ coincides with the itinerary $q_0', \, \ldots \, q_p'$.
In the second case, these itineraries satisfy $q_k' = q_k$ for $1 \leq k \leq s_0 + l' - 1$ and $q_{k+1}' = q_k$
for $s_0 + l' \leq k \leq l$. In other words, the two itineraries coincide up to the point $q_{s_0 + l'}'$. This
point is, in fact, ``fictitious'' for the itinerary in question in the sense that
$[W_3 (a,b)]^{-1} \ast W_4 (a,b) \ast W_3 (a,b)$ is not
written in a reduced way and itineraries should be defined only in this case. In fact, with the above notations, taking $W_3
= \vartheta_{s_0}^{r_{i_0}} \ast \cdots \ast \vartheta_1^{r_1}$ and $W_4 = \vartheta_{l'+s_0}^{r_{l'+s_0}} \ast \cdots \ast \vartheta_{s_0+1}^{r_{s_0+1}}$, we have that $\vartheta_{l'+s_0}$ and $\vartheta_{s_0}^{-1}$ take on same same value if $W_4$
does not coincide with the minimal conjugate of type~$1$ of $W$.

Item~(3) of Lemma~\ref{herewego4} ensures that the elements $q = \tilde{q}_0, \tilde{q}_1 \ldots , \tilde{q}_{s_0+ l'-1}$
in the itinerary of $q$ under $W (\tilde{h}_{1}^{-1} \circ f \circ \tilde{h}_{1}, \, \tilde{h}_{2}^{-1}
\circ g \circ \tilde{h}_{2})$ are pairwise distinct. The ``fictitious'' point is not included in this list and this
will play a role in the proof of Lemma~\ref{herewego3}.

To not interrupt the discussion, let us first conclude the proof of Lemma~\ref{herewego3} before proving
Lemma~\ref{herewego4}.

\begin{proof}[Proof of Lemma~\ref{herewego3}]
Let us begin by summarizing what have already been mentioned concerning the proof of this lemma.

Let $W_{i} (a,b) = \vartheta_{l}^{r_l} \ast \cdots \ast \vartheta_1^{r_1}$ and $W_{j} (a,b) = \vartheta_{m}^{s_m}
\ast \cdots \ast \vartheta_1^{s_1}$ and assume that $m \leq l$. The construction of the required perturbation $(\overline{h}_1,
\overline{h}_2)$ will be carried out by induction on $m$. It has already been mentioned that the statement is true for $m=1$.
Assume that the existence of the desired perturbation $(\overline{h}_1, \overline{h}_2)$ was already established
for $1, \ldots ,m-1$ and every $l \in \N$. Let us now construct $(\overline{h}_1, \overline{h}_2)$ arbitrarily close to
$(h_1,h_2)$ and such that $q$ remains fixed by exactly one of the elements
$W_i (\overline{h}_1, \overline{h}_2)$ and $W_j (\overline{h}_1, \overline{h}_2)$.
The result will immediately follow in this case.

Let $W_4 (a,b)$ denote the minimal conjugate of $W_{i} (a,b)$ and set
\[
W_{i} (a,b) = [W_3 (a,b)]^{-1} \ast W_4 (a,b) \ast W_3 (a,b) \, ,
\]
where $W_3 (a,b)$ is empty if $W_{i} (a,b)$ coincides with its minimal conjugate. The length of the minimal conjugate $W_4 (a,b)$
is going to be denoted by $l'$. If $l'<l$, then $s_0$ will denote the length of the word $W_3 (a,b)$. The induction assumption allows
us to suppose that $l' \geq m$, otherwise we can conjugate the whole group by $W_3 (a,b)$ and the problem
will be reduced to eliminate the corresponding common fixed point between
$W_4 ( \tilde{f}, \tilde{g})$ and another word $\overline{W} (\tilde{f}, \tilde{g})$ which
which is conjugate to $W_{j} (\tilde{f}, \tilde{g})$. Since the length of $W_4 (a,b)$ is~$l'< m$ the induction
assumption implies that the common fixed point in question can effectively be eliminated.
Similarly it is clear that the word $W_{j} (a,b)$, of length~$m$, is not a conjugate of a
shorter word since otherwise the statement results immediately.

In view of the preceding, in the sequel we always have $l'\geq m$. The proof is divided in three cases.
First we will assume that $l'$ is strictly greater than $m$, i.e. $l' > m$. When $l'=m$ two further cases need to
be considered, according to whether or not $W_3(a,b)$ is void, i.e. according to $s_0 \geq 1$ or $s_0 = 0$.

{\bf Case 1.} Suppose first that $l' > m$ strictly. This case will be handled with a useful general observation. By using
Lemma~\ref{herewego4}, the points $q = q_0^1, \ldots , q_{s_0+ l'-1}^1$ can be supposed pairwise distinct. Also, it is clear
that the points $q_{s_0+1}^1, \ldots , q_{s_0+ l'-1}^1$ appears exactly once in the (full) itinerary of $q$ under $W_{i}
(\tilde{f}, \tilde{g})$. Since $l' > m$ and $q_0^2 = q_m^2 = q_0^1 \not \in
\{q_{s_0+1}^1, \ldots, q_{s_0+l'-1}^1\}$, it follows the existence of a point
$q_N^1$, $i_0 < N < k_0 + l'$, which does not
belong to the itinerary of $q$ under $W_{j} (h_{1}^{-1} \circ f \circ h_{1}, \, h_{2}^{-1} \circ g \circ h_{2})$, i.e. to
the set $\{ q_0^2, \ldots , q_m^2 \}$. Without loss of generality, the value of $\vartheta_{m+1}$ can be supposed to be $a$.
In this case, we shall consider perturbations $h_{1,t}$ of $h_1$ given by
\begin{equation}
h_{1,t} (z) = h_1 (z) + t z^{\alpha +1} P(z)  \label{almostfinalperturbations}
\end{equation}
where $P$ is a polynomial vanishing at all points $q_k^1$, $k \neq N$ and $k \in \{ 0, \ldots ,l\}$, and over all points $q_k^2$,
$k \in \{ 0, \ldots , m\}$. Yet, $P$ verifies $P (q_N^1) \neq 0$. Let then $h_{1,\ast} = h_{1,t}$, for some sufficiently small
$t >0$, and $h_{2,\ast} = h_{2}$. Clearly $W_{j} (h_{1,\ast}^{-1} \circ f \circ h_{1,\ast}, \, h_{2,\ast}^{-1} \circ g \circ
h_{2,\ast}) (q) =q$. It remains to check that $W_{i} (h_{1,\ast}^{-1} \circ f \circ h_{1,\ast}, \, h_{2,\ast}^{-1} \circ g \circ
h_{2,\ast}) (q) \neq q$. Setting $W_{i} (a,b) = W_B (a,b) \ast W_A (a,b)$ where $W_A (a,b) =  \vartheta_{N+1}^{r_{N+1}} \ast
\cdots \ast \vartheta_1^{r_1}$ and $W_B (a,b) = \vartheta_{l}^{r_l} \ast \cdots \ast \vartheta_{N+2}^{r_{N+2}}$ (being $W_B (a,b)$
possibly empty), by construction, we have: $W_B (h_{1,\ast}^{-1} \circ f \circ h_{1,\ast}, \, h_{2,\ast}^{-1} \circ g \circ h_{2,\ast})
(q_{N+1}^1) = q_l^1 = q_0^1$ whereas $W_A (h_{1,\ast}^{-1} \circ f \circ h_{1,\ast}, \, h_{2,\ast}^{-1} \circ g \circ h_{2,\ast})
(q_{0}^1) \neq q_{N+1}^1$. Since all the maps involved are one-to-one, we conclude that $W_{i} (h_{1, \ast}^{-1} \circ f \circ
h_{1, \ast}, \, h_{2,\ast}^{-1} \circ g \circ h_{2,\ast}) (q) \neq q$ and the statement is proved in this first case.

Note that in the construction above it was implicitly used the fact that $q^{1}_N$, the element in the itinerary of $q$ under
$W_{i} (h_{1}^{-1} \circ f \circ h_{1}, \, h_{2}^{-1} \circ g \circ h_{2})$ which does not belong to the itinerary of $q$
under $W_{j} (h_{1}^{-1} \circ f \circ h_{1}, \, h_{2}^{-1} \circ g \circ h_{2})$, does not correspond to the above mentioned ``fictitious'' point. In fact, if this were the case, then both $\vartheta_{l'+s_0}$ in $W_4$ and $\vartheta_{s_0}^{-1}$
in $W_3^{-1}$ would take on same value and, therefore, the effect of $P$ would be void.

{\bf Case 2.} Suppose now that $l'=m$. The preceding argument can then easily be adapted to handle the case where $s_0 \geq 1$.
The main difference between the present case and the above discussion lies in the fact that $q_N^i$, the element in the itinerary
of $q$ under $W_{i} (h_{1}^{-1} \circ f \circ h_{1}, \, h_{2}^{-1} \circ g \circ h_{2})$ which does not belong to the itinerary
of $q$ under $W_{j} (h_{1}^{-1} \circ f \circ h_{1}, \, h_{2}^{-1} \circ g \circ h_{2})$, may appear ``twice'' in the (full)
itinerary of $q$ under $W_{i} (h_{1}^{-1} \circ f \circ h_{1}, \, h_{2}^{-1} \circ g \circ h_{2})$. In other words, $N$
belongs to $\{1, \ldots, s_0+l'-1\}$ and not necessarily to $\{s_0+1, \ldots, s_0+l'-1\}$ as in the previous case. Let us
present the adaptations required to establish the result in this context.

If $q_N^1$ appears only once in the full itinerary of $q$, then the proof follows as above. So, let us assume that $q_N^1$ appears twice and that the value of $\vartheta_{N+1}$ is $a$.
Let us first consider perturbations $h_{1,t}$ of $h_1$ given by
\[
h_{1,t}(z) = h_1(z) + tz^{\alpha +1} P(z)
\]
where $P$ is a polynomial vanishing at all points $q_k^1$, $k \ne N$ and $k \in \{0,\ldots, l\}$ and over all points $q_k^2$,
$k \in \{0, \ldots, m\}$. Furthermore $P(q_N^1)$ should be different from zero. Let then $h_{1,\ast} = h_{1,t}$, for some
sufficiently small $t > 0$, and $h_{2,\ast} = h_2$. Clearly $W_{j} (h_{1,\ast}^{-1} \circ f \circ h_{1,\ast}, \,
h_{2,\ast}^{-1} \circ g \circ h_{2,\ast}) (q) = q$. Nonetheless we cannot yet ensure that $W_{i} (h_{1,\ast}^{-1}
\circ f \circ h_{1,\ast}, \, h_{2,\ast}^{-1} \circ g \circ h_{2,\ast}) (q) \ne q$. To begin with, let us describe when
$q$ is a fixed point for $W_{i} (h_{1,\ast}^{-1} \circ f \circ h_{1,\ast}, \, h_{2,\ast}^{-1} \circ g \circ h_{2,\ast})$.

We are assuming that $q_N^1$ appears twice in the itinerary of $q$ so that
$q_N^1$ coincides with $q_{l-N}^1$. Without loss of generality, the value of
$\vartheta_{N+1}$ can be supposed to be $a$. If $\vartheta_{l-N+1}$ takes on value $b$ then the argument presented in the
previous case ensures that $q$ is not a fixed point for $W_{i} (h_{1,\ast}^{-1} \circ f \circ h_{1,\ast}, \, h_{2,\ast}^{-1}
\circ g \circ h_{2,\ast})$. So, it can be assumed that $\vartheta_{l-N+1}$ also takes on value $a$. It can easily be checked that
$q$ is a fixed point for $W_{i} (h_{1,\ast}^{-1} \circ f \circ h_{1,\ast}, \, h_{2,\ast}^{-1} \circ g \circ h_{2,\ast})$ if
and only if
\begin{equation}\label{cond_fixedpoint}
h_{1,t} \circ \vartheta_{l-N}^{r_{l-N}} \circ \cdots \circ \vartheta_{N+1}^{r_{N+1}} (q_N^1) = h_1(q_N^1)
\end{equation}
where $\vartheta_{k}$ is substituted by $h_{1,t}^{-1} \circ f \circ h_{1,t}$ (resp. $h_{2,t}^{-1} \circ g \circ h_{2,t}$) if
$\vartheta_k$ takes on value $a$ (resp. $b$). Therefore, to make sure that $q$ is
not a fixed point for $W_{i} (h_{1,\ast}^{-1} \circ
f \circ h_{1,\ast}, \, h_{2,\ast}^{-1} \circ g \circ h_{2,\ast})$, we should guarantee that $h_{1,t}^{-1}\circ h_1(q_N^1) \ne
\vartheta_{l-N}^{r_{l-N}} \circ \cdots \circ \vartheta_{N+1}^{r_{N+1}} (q_N^1)$. Let us assume that $h_{1,t}^{-1}\circ h_1(q_N^1)
= \vartheta_{l-N}^{r_{l-N}} \circ \cdots \circ \vartheta_{N+1}^{r_{N+1}} (q_N^1)$. Let $Q$ be a polynomial vanishing at all points
$q_k^1$, $k \in \{0,\ldots, l\}$ and over all points $q_k^2$, $k \in \{0, \ldots, m\}$. Let also $Q$ verify $Q(h_{1,t}^{-1}\circ
h_1(q_N^1)) \ne 0$. Now, since~$t$ is arbitrary, it can be chosen so that $h_{1,t}^{-1}\circ h_1(q_N^1) \not \in \{q_0^1,
\ldots, q_l^1\} \cup \{q_0^2, \ldots, q_m^2\}$. Consider the perturbations of $h_1$ given by
\[
\overline{h}_{1,t} = h_{1,t} + tz^{\alpha +1}Q(z) \, .
\]
Let $\overline{h}_{1,\ast} = \overline{h}_{1,t}$, for some sufficiently small $t > 0$, and $\overline{h}_{2,\ast} = h_2$.
Clearly $W_{j} (h_{1,\ast}^{-1} \circ f \circ h_{1,\ast}, \, h_{2,\ast}^{-1} \circ g \circ h_{2,\ast}) (q) = q$. Moreover,
since $Q(h_{1,t}^{-1} \circ h_1(q_N^1)) \ne 0$, Condition~(\ref{cond_fixedpoint}) is no longer verified for $(\overline{h}_{1,\ast},
\overline{h}_{2,\ast})$ and, therefore, $W_i (h_{1,\ast}^{-1} \circ f \circ h_{1,\ast}, \, h_{2,\ast}^{-1} \circ g \circ
h_{2,\ast}) (q) \ne q$. The statement is proved in this case.

{\bf Case 3.} Summarizing what precedes, it only remains to deal with words $W_{i} (a,b), \, W_{j} (a,b)$ of same length
(i.e. $l=m$) and none of them being a conjugate of a shorter word. In particular, both words $W_{i} (a,b), \, W_{j} (a,b)$
are on the same ``footing'' concerning the argument below. Next note that, since none of them is conjugate to a shorter word,
Lemma~\ref{herewego4} allows us to suppose without loss of generality that the points $q=q_0^1, \ldots , q_{l-1}^1$ (resp.
$q=q_0^2, \ldots , q_{m-1}^2$) are pairwise distinct. In particular, unless these two itineraries coincide up to relabeling
the points, there exists again $q_N^1$, $0 < N <l$, which does not belong to the set $\{ q_0^2, \ldots , q_m^2 \}$. Thus the
preceding argument can still be employed to settle the lemma.

Finally, we only have to deal with the case where the itineraries of $q$ under $W_{i} (h_{1}^{-1} \circ f \circ h_{1}, \,
h_{2}^{-1} \circ g \circ h_{2})$ and under $W_{j} (h_{1}^{-1} \circ f \circ h_{1}, \, h_{2}^{-1} \circ g \circ h_{2})$
coincide up to relabeling the points. The discussion splits in three subcases.

\noindent {\it Case 3a.} Suppose that in the word $W_{i} (a,b)$, $\vartheta_1$ takes on the value $a$ whereas, in the word
$W_{j} (a,b)$, $\vartheta_1$ takes on the value~$b$ (or the other way around). Consider then perturbations $h_{1,t}$ of
$h_1$ having the form~(\ref{almostfinalperturbations}), where, this time, $P (q) \neq 0$ and $P$ vanishes over all the
remaining points of the common itinerary $q = q_0^1, \ldots , q_{l-1}^1$. Again we set $h_{1,\ast} = h_{1,t}$, for small
$t >0$, and $h_{2,\ast} = h_{2}$. Since $q_k^2 \neq q_0$ for every $k=1, \ldots , k-1$, it follows that $W_{j} (h_{1,\ast}^{-1}
\circ f \circ h_{1,\ast}, \, h_{2,\ast}^{-1} \circ g \circ h_{2,\ast}) (q) =q$. On the other hand, the same argument employed
above shows that $W_{i} (h_{1}^{-1} \circ f \circ h_{1}, \, h_{2}^{-1} \circ g \circ h_{2}) (q) \neq q$ and finishes the proof
of the proposition in the present case.

\noindent {\it Case 3b.} Suppose that in both words $W_{i} (a,b), \, W_{j} (a,b)$, $\vartheta_1$ takes on the same value,
say~$a$, but we have $r_1 \neq s_1$ (where the exponents are allowed to be negative). Since $0 \in \C$ is not a Cremer
point for $f$, it follows that $q_1^1 \neq q_1^2$. However $q_1^2$ must belong to the itinerary of $q$ under $W_{i} (h_{1}^{-1}
\circ f \circ h_{1}, \, h_{2}^{-1} \circ g \circ h_{2})$. Let then $q_N^1 = q_1^2$ where $N \in \{ 2, \ldots , m-1\}$ (recall
that $l=m$). If $N < m-1$ consider the word $W^{(1)} (a,b) = W_B \ast W_A$ where $W_B (a,b) = a^{-s_1}$ and $W_A (a,b) =
\vartheta_{m}^{r_m} \ast \cdots \ast \vartheta_1^{r_1}$, where $\vartheta_k$ takes on the same value taken in the spelling
of $W_{i} (a,b)$, for $k = 1, \ldots ,N$. It follows that $W^{(1)} (a,b)$ is a word of length less than~$m$ satisfying
$W^{(1)} (h_{1}^{-1} \circ f \circ h_{1}, \, h_{2}^{-1} \circ g \circ h_{2}) (q) = q$. Hence, by induction assumption, we
can destroy the fact that $q$ is a common fixed point for $W^{(1)} (h_{1,\ast}^{-1} \circ f \circ h_{1,\ast}, \,
h_{2,\ast}^{-1} \circ g \circ h_{2,\ast})$ and $W_{i_2} (h_{1,\ast}^{-1} \circ f \circ h_{1,\ast}, \, h_{2,\ast}^{-1} \circ g
\circ h_{2,\ast})$.

Naturally we can assume that $W_{i} (h_{1,\ast}^{-1} \circ f \circ h_{1,\ast}, \, h_{2,\ast}^{-1} \circ g \circ h_{2,\ast}), \,
W_{j} (h_{1,\ast}^{-1} \circ f \circ h_{1,\ast}, \, h_{2,\ast}^{-1} \circ g \circ h_{2,\ast})$ still have a common fixed point
in $B (\delta)$, otherwise the statement is established. Let us still denote by $q$ this common fixed point. We claim that
now the itineraries of $q$ under $W_{i} (h_{1,\ast}^{-1} \circ f \circ h_{1,\ast}, \, h_{2,\ast}^{-1} \circ g \circ h_{2,\ast})$
and under $W_{j} (h_{1,\ast}^{-1} \circ f \circ h_{1,\ast}, \, h_{2,\ast}^{-1} \circ g \circ h_{2,\ast})$ do not coincide,
even though these points are relabeled. In fact, denote by
$q = \tilde{q}_0^1, \ldots, \tilde{q}_m^1$ (resp. $q = \tilde{q}_0^2, \ldots, \tilde{q}_m^2$)
the itinerary of $q$ under $W_{i} (h_{1,\ast}^{-1} \circ f \circ h_{1,\ast}, \, h_{2,\ast}^{-1} \circ g \circ h_{2,\ast})$
(resp. $W_{j} (h_{1,\ast}^{-1} \circ f \circ h_{1,\ast}, \, h_{2,\ast}^{-1} \circ g \circ h_{2,\ast})$). The perturbation
$h_{1,\ast}$ can be chosen so that
$$
\Vert q_k^1 - \tilde{q}_1^1 \Vert  < \tau/4 \; \; \, {\rm and} \; \; \,
\Vert q_k^2 - \tilde{q}_1^2 \Vert  < \tau/4
$$
for all $k=1, \ldots, m$, where $\tau$
denotes a lower bound for the distance between two distinct points on each itinerary of $q$ under $W_{i} (h_{1,\ast}^{-1}
\circ f \circ h_{1,\ast}, \, h_{2,\ast}^{-1} \circ g \circ h_{2,\ast})$, $W_{j} (h_{1,\ast}^{-1} \circ f \circ h_{1,\ast},
\, h_{2,\ast}^{-1} \circ g \circ h_{2,\ast})$. Assume for a contradiction that the itineraries $q = \tilde{q}_0^1, \ldots,
\tilde{q}_m^1$ and $q = \tilde{q}_0^2, \ldots, \tilde{q}_m^2$ coincide up to relabeling the points. The estimate above guarantees
that $\tilde{q}_1^2$ again coincides with $\tilde{q}_N^2$. Therefore $q$ is a common fixed point for $W^{(1)} (h_{1,\ast}^{-1}
\circ f \circ h_{1,\ast}, \, h_{2,\ast}^{-1} \circ g \circ h_{2,\ast})$ and $W_{j} (h_{1,\ast}^{-1} \circ f \circ h_{1,\ast},
\, h_{2,\ast}^{-1} \circ g \circ h_{2,\ast})$, contradicting the choice of $h_{1,\ast}$. Thus, we have
just proved that
$\tilde{q}_N^1 \not \in \{\tilde{q}_0^2, \ldots, \tilde{q}_m^2\}$. Therefore
the same arguments employed in
the previous cases can now be brought to settle the lemma in the present case as well.

Let us finally consider the case $N=m-1$. We are going to show that the induction assumption can also be used to eliminate
this situation as well. This goes as follows. Set $W^{(1)} (a,b)= W_B \ast W_A$ where $W_A (a,b) = a^{r_1 - s_1}$ and
$W_B (a,b) = \vartheta_m^{r_m} \ast \cdots \ast \vartheta_2^{r_2}$. Then $W^{(1)} (a,b)$ is a word of length~$m-1$ such that
$W^{(1)} (h_{1}^{-1} \circ f \circ h_{1}, \, h_{2}^{-1} \circ g \circ h_{2}) (q_1^2) =q_1^2$. On the other hand $q_1^2$ is also
fixed by $W^{(2)} (h_{1}^{-1} \circ f \circ h_{1}, \, h_{2}^{-1} \circ g \circ h_{2})$ where $W^{(2)}(a,b)$ is the word of
length $m$ obtained from $W_{j}$ through conjugation under $a^{-s_1}$, i.e. $W^{(2)}(a,b) = a^{s_1} \ast W_{j}(a,b) \ast
a^{-s_1}$. Since one of this words has length strictly smaller then $m$, the induction assumption allows us to eliminate
the common fixed point $q_1^2$. Therefore, if $W_{i} (h_{1,\ast}^{-1} \circ f \circ h_{1,\ast}, \, h_{2,\ast}^{-1} \circ
g \circ h_{2,\ast})$ and $W_{j} (h_{1,\ast}^{-1} \circ f \circ h_{1,\ast}, \, h_{2,\ast}^{-1} \circ g \circ h_{2,\ast})$
still have a common fixed point, again we claim that the itineraries of $q$ under $W^{(1)} (h_{1,\ast}^{-1} \circ f \circ
h_{1,\ast}, \, h_{2,\ast}^{-1} \circ g \circ h_{2,\ast})$ and under
$W_{j} (h_{1,\ast}^{-1} \circ f \circ h_{1,\ast}, \,
h_{2,\ast}^{-1} \circ g \circ h_{2,\ast})$ do not coincide in the sense that $\tilde{q}_N^1$ does not belong to the
itinerary of $q$ under $W_{j} (h_{1,\ast}^{-1} \circ f \circ h_{1,\ast}, \, h_{2,\ast}^{-1} \circ g \circ h_{2,\ast})$.
Now the same argument can again be applied to eliminate this common fixed point.
The lemma is proved in subcase 3b.

\noindent {\it Case 3c.} Suppose that in both words $W_{i} (a,b), \, W_{j} (a,b)$, $\vartheta_1$ takes on the same
value, say~$a$, and that $r_1 = s_1$. In this case we have $q_1^1 =q_1^2$. This case amounts to consider the new words
$\widetilde{W}_{i} (a,b)$ and $\widetilde{W}_{j} (a,b)$ of length $m$ and given respectively by $a^{r_1} \ast
\vartheta_{m}^{r_m} \ast \cdots \ast \vartheta_2^{r_2}$ and $a^{r_1} \ast \vartheta_{m}^{s_m} \ast \cdots \ast
\vartheta_2^{s_2}$. Clearly $q_1^1=q_1^2$ is a common fixed point for these words.  Note that $\vartheta_2$ takes
on the same value $b$. However unless $r_2=s_2$, this common fixed point can be destroyed by means of the preceding
arguments. Otherwise, we shall have $q_2^1 =q_2^2$ and we may move on to the new words spelled out respectively as $b^{r_2}
\ast a^{r_1} \ast \vartheta_{m}^{r_m} \ast \cdots \ast \vartheta_3^{r_3}$ and
as $b^{r_2} \ast a^{r_1} \ast \vartheta_{m}^{s_m}
\ast \cdots \ast \vartheta_3^{s_3}$ which have $q_1^2 =q_2^2$ as common fixed points. Continuing in this way, since $W_{i}
(a,b), \, W_{j} (a,b)$ are incommensurable, at some step or order $N$ we must have $r_N \neq s_N$. At this
moment the common
fixed point can be destroyed and the rest of the statement will quickly follows. This ends the proof of Lemma~\ref{herewego3}.
\end{proof}

To finish the paper, the last step is to supply the proof of Lemma~\ref{herewego4}.

\begin{proof}[Proof of Lemma~\ref{herewego4}].
The statement can naturally be proved by inducting on the length $l$ of $W (a,b)$. To begin with, recall that $W (a,b)$ is
spelled out as $W (a,b) = \vartheta_{l}^{r_l} \ast \cdots \ast \vartheta_1^{r_1}$. Consider the itinerary $q=q_0 , \ldots,
q_l$ of $q$ under $W (\tilde{f}, \tilde{g})$. Let us first consider the situations described in items~(1) or~(2).
Then the statement amounts to showing
that, modulo perturbing $(h_1,h_2)$ into certain $(\tilde{h}_1, \tilde{h}_2)$, arbitrarily close to $(h_1,h_2)$,
we obtain an itinerary $q=\tilde{q}_0 , \ldots , \tilde{q}_l$ such that the points
$q=\tilde{q}_0 , \ldots , \tilde{q}_{l-1}$ are pairwise distinct. The idea is to use a perturbation
similar to the one employed in Lemma~\ref{directuse}
to make the mentioned points pairwise distinct and then to ``locally correct'' it by adding a new localized perturbation that
will ensure the points $q_0, \, q_l$ do not move at the final situation.

Let then $\epsilon > 0$ be fixed. We shall look for $(\tilde{h}_1, \tilde{h}_2)$ $\epsilon$-close to $(h_1, h_2)$ satisfying
the required conditions where, by saying that $(\tilde{h}_1, \tilde{h}_2)$ is $\epsilon$-close to $(h_1, h_2)$, it is meant
that $\max \{ d_A (\tilde{h}_1, h_1), \, d_A (\tilde{h}_2, h_2)\} < \epsilon$.

To fix notations, assume also that $\vartheta_1$ takes on the value~$a$ since the other possibility is totally analogous (throughout
the discussion negative exponents are allowed). Let us first suppose that $q=q_0 \neq q_l$. Because $q_0$ is supposed to be different
from each of the points $q_1, \ldots , q_{l-1}$, cf. Lemma~\ref{minorlemma2}, in addition to $q_l$,
there is $\tau >0$ so that, for every $k \in \{ 1, \dots ,l\}$
we have $\vert q_0 -q_k\vert > 2\tau$. In other words, the distances of the mentioned points $q_k$ to $q_0$
is bounded from below by a positive constant.

Now consider the word $\overline{W} (a,b)$ which is nothing but $W (a,b)$ spelled out in the ``reverse order and with exponents
of opposite signs'' (i.e. $\overline{W} (a,b)$ is the inverse element of $W (a,b)$). In particular, $q_l , \ldots , q_0$ becomes
the itinerary of $q_l$ under
$\overline{W} (h_{1}^{-1} \circ f \circ h_{1}, \, h_{2}^{-1} \circ g \circ h_{2}) = \overline{W} (\tilde{f}, \tilde{g})$
whose ``final point'' is $q_0$. By successively applying Lemma~\ref{directuse}, we can find an arbitrarily
small perturbation $(h_{1,\ast}, h_{2,\ast})$ of $(h_1, h_2)$ so that the itinerary $q_l = q_l', q_{l-1}', \ldots , q_0'$
of $q_l$ under $\overline{W} (h_{1,\ast}^{-1} \circ f \circ h_{1,\ast}, \, _{2,\ast}^{-1} \circ g \circ h_{2},\ast)$
is constituted by pairwise distinct points. In addition, we can assume that $(h_{1,\ast}, h_{2,\ast})$ is
$\epsilon/2$-close to $(h_1, h_2)$ whereas, for every $k=0, \ldots , l$, the distance $\vert q_k -q_k'\vert$
is arbitrarily small. In particular we have
\begin{equation}\label{estimateforlagrange}
\vert q_0 -q_0' \vert < \tau \, .
\end{equation}
Modulo choosing $(h_{1,\ast}, h_{2,\ast})$ closer to $(h_1, h_2)$, we are going to construct a new perturbation $\tilde{h}_1$
of $h_{1,\ast}$, verifying $d_A (\tilde{h}_1, h_{1,\ast}) < \epsilon/2$, and such that $(\tilde{h}_1,
\tilde{h}_2) = (\tilde{h}_1, h_{2,\ast})$ satisfies the required conditions.
To construct $\tilde{h}_1$ consider the elementary Lagrange
interpolation consisting of a Polynomial $P$ (of degree at most $l$) satisfying the following conditions:
\begin{enumerate}
\item $P (q_1') = \cdots = P (q_l') =0$.

\item $P (q_0) = h_{1,\ast} (q_0') - h_{1,\ast} (q_0)$.
\end{enumerate}
Thanks to Estimate~(\ref{estimateforlagrange}), it follows that all coefficients of $P$ converge to $0$
provided that all the distances
$\vert q_0 -q_0'\vert$ do so. Note that Estimate~(\ref{estimateforlagrange}) guarantees that
$q_0'$ does not
coincide with any one of the elements $q_1', \ldots, q_l'$. Since the degree of $P$
is uniformly bounded, it
follows that $P$ converges in the analytic topology for the null function. In particular,
modulo taking $(h_{1,\ast}, h_{2,\ast})$
sufficiently close to $(h_1, h_2)$ in the analytic topology, we conclude that $\tilde{h}_1 = h_{1,\ast}
+P$ is $\epsilon/2$-close to $h_{1,\ast}$ and, hence, $\epsilon$-close to $h_1$. Finally it is clear
that the itinerary of $q_l =q_l'$ under
$\overline{W} (\tilde{h}_{1}^{-1} \circ f \circ \tilde{h}_{1}, \, \tilde{h}_{2}^{-1} \circ g \circ \tilde{h}_{2})$
is simply $q_l=q_l', q_{l-1}', \ldots , q_1', q_0$. The lemma is proved in the first case.

Consider now the case where $q_0 = q_l$. The word $W (a,b)$ is first supposed not to be conjugate to a shorter word.
Recalling that $\vartheta_1$ takes on the value~$a$, we conclude that $\vartheta_l$ takes on the value~$b$. The proof
then amounts to noticing that the same argument above applies: a first perturbation $(h_{1,\ast},
h_{2,\ast})$ is
constructed so that the itinerary $q_l =q_l', q_{l-1}', \ldots , q_0'$ of $q_l$ under $\overline{W}
(h_{1,\ast}^{-1} \circ f \circ h_{1,\ast}, \, h_{2,\ast}^{-1} \circ g \circ h_{2,\ast})$ is constituted by
pairwise distinct points.
To make $q_0'$ to coincide with $q_0 =q_l$, just note that
$q_{l-1}'$ is not affected by the replacing of $h_{1,\ast}$ by $\tilde{h}_1$, since
$\vartheta_l$ takes on the value $b$.
This settles the second item in the statement of the lemma.

Finally let us consider the case where $W (a,b)$ is conjugate to a shorter word. Set $W (a,b) = [W_3 (a,b)]^{-1} \ast
W_4 (a,b) \ast W_3 (a,b)$ where $W_4 (a,b)$ is the minimal conjugate of $W (a,b)$. This case is nothing but a blend of
item~($1$), applied to $W_3 (a,b)$ and of item~($2$) applied to $W_4 (a,b)$, details are left to the reader.
\end{proof}

\bigskip

\begin{flushleft}
{\sc Julio Rebelo} \\
Institut de Math\'ematiques de Toulouse\\
118 Route de Narbonne\\
F-31062 Toulouse, FRANCE.\\
rebelo@math.univ-toulouse.fr

\end{flushleft}

\bigskip

\begin{flushleft}
{\sc Helena Reis} \\
Centro de Matem\'atica da Universidade do Porto, \\
Faculdade de Economia da Universidade do Porto, \\
Portugal\\
hreis@fep.up.pt \\

\end{flushleft}


\begin{thebibliography}{Dillo 83}

\bibitem[BLL-1]{Frank1} {\sc M, Belliart, I. Liousse \& F. Loray}, Sur l'existence de points fixes attractifs pour les sous-groupes de
${\rm Aut}\, (\C, 0)$, {\it C. R. Acad. Sci. Paris}, {\bf 324}, S\'er. {\bf I}, (1997), 443-446.

\bibitem[BLL-2]{Morefrank} {\sc M, Belliart, I. Liousse \& F. Loray}, The generic differential equation $dw/dz = P_n (w,z) /Q_n (w,z)$ on
$\C P(2)$ carries no interesting transverse structure, {\it Ergod. Th. \& Dynam. Sys.}, {\bf 21}, (2001), 1599-1607.

\bibitem[C-G]{carleson} {\sc L. Carleson \& T. Gamelin}, {\it Complex Dynamics},
Springer-Verlag, New York, (1993).


\bibitem[GM-W]{gomez-mont} {\sc X. Gomez-Mont \& B. Wirtz}, On fixed points of conformal pseudogroups, {\it Bull. Braz. Math. Soc.},
{\bf 26}, 2, (1995), 201-209.

\bibitem[Il-1]{ilyaso} {\sc Yu. Il'yashenko}, Topology of phase portraits of analytic differential equations on a complex projective plane,
{\it Trudy Sem. Petrovsk.}, {\bf 4}, (1978), 83-136.

\bibitem[Il-2]{ilyaso2} {\sc Yu. Il'yashenko}, Global and local aspects of the theory of complex differential equations,
{\it Proc. Int. Cong. Math. Helsinki}, Acad. Scient. Fennica, {\bf 2}, (1978), 821-826.

\bibitem[K-H]{katok} {\sc A. Katok \& B. Hasselblatt}, {\it Introduction to the Modern Theory of Dynamical Systems},
Encyclopedia of Mathematics and its Applications, Volume 54, Cambridge University Press, (1997).

\bibitem[LF]{lefloch} {\sc L. Le Floch}, Rigidit\'e g\'en\'erique des feuilletages {\it Ann. Sc. de l'ENS, S\'er. 4},
{\bf 31}, 6, (1998), 765-785.

\bibitem[Lo]{frank} {\sc F. Loray}, {\it Pseudo-groupe d'une singularit\'e de feuilletage holomorphe en dimension deux},
{\it available from hal.archives-ouvertes.fr/hal-00016434}, (2005).

\bibitem[M-M]{marinmattei} {\sc D. Marin \& J.-F. Mattei}, Incompressibilit\'e des feuilles de germes de
feuilletages holomorphes singuliers, {\it Ann. Scient. de l'ENS, 4-S\'erie}, {\bf 41}, (2008), 855-903.

\bibitem[M-R-R]{MRR} {\sc J.-F. Mattei, J.C. Rebelo \& H. Reis}, Generic pseudogroups on $(\C ,0)$ and the topology of leaves, {\it to appear
in Compositio Mathematica}.

\bibitem[M-S]{matteisalem} {\sc J.-F. Mattei \& E. Salem}, Modules formels locaux de feuilletages
holomorphes, available from arXiv:math/0402256v1, (2004), 89 pages.

\bibitem[N]{nakai} {\sc I. Nakai}, Separatrizes for non solvable dynamics on $(\C, 0)$,
{\it Ann. Inst. Fourier}, {\bf 44}, (1994), 569-599.


\bibitem[Sh]{scherba} {\sc A. Shcherbakov}, On the density of an orbit of a pseudogroup of conformal mappings and a generalization
of the Hudai-Verenov theorem, {\it Vestinik Movskovskogo Universiteta Mathematika}, {\bf 31}, 4, (1982), 10-15.

\bibitem[Sh-R-O]{scherbaRO} {\sc A. Shcherbakov, E. Rosales-González \& L. Ortiz-Bobadilla}, Countable set of limit cycles for the
equation $dw/dz = P_n(z,w)/Q_n(z,w)$, {\it J. Dynam. Control Systems}, {\bf 4}, 4, (1998), 539-581.

\bibitem[T]{takens} {\sc F. Takens}, A nonstabilizable jet of a singularity of a vector field: the analytic case,
{\it Algebraic and differential topology - global differential geometry}, Teubner-Text Math., {\bf 70}, Teubner, Leipzig, (1984), 288-305.


\bibitem[Y]{perezmarco} {\sc J.-C. Yoccoz}, Centralisateurs et conjugaison diff\'erentiable des diff\'eomorphismes
du cercle. Petits diviseurs en dimension~1, {\it Ast\'erisque}, {\bf 231}, (1995), 89-242.

\end{thebibliography}
\end{document}